\documentclass[a4paper,10pt]{amsart}

\usepackage{amssymb,amsmath,amsthm,mathrsfs,enumerate,graphicx, color}
\usepackage[pdfpagelabels,colorlinks,linkcolor=blue,citecolor=black,urlcolor=blue]{hyperref}
\usepackage{esint}
\newtheorem{thm}{Theorem}[section]

\newtheorem{lem}[thm]{Lemma}
\newtheorem{prop}[thm]{Proposition}

\newtheorem{rem}[thm]{Remark}

\newcommand{\lesi}{\lesssim}

\newcommand{\f}{\frac}

\newcommand{\su}{\subset}
\newcommand{\vc}{\infty}
\newcommand{\Rd}{\mathbb{R}^d}

\newcommand{\La}{\mathcal{L}_a}

\begin{document}

\title[Smoothing estimates for Schr\"odinger operators with inverse-square potentials]
  {Weighted estimates for powers and Smoothing estimates of Schr\"odinger operators with inverse-square potentials}

\authors

\author[T. A. Bui]{The Anh Bui}
\address{The Anh Bui \\
Department of Mathematics, Macquarie University, NSW 2109, Australia}
\email{the.bui@mq.edu.au, bt\_anh80@yahoo.com}

\author[P. D'Ancona]{Piero D'Ancona}
\address{Piero D'Ancona\\ Dipartimento di Matematica, Sapienza -- Universit\`{a} di Roma, Piazzale A.~Moro 2, 00185 Roma, Italy}
\email{dancona@mat.uniroma1.it}

\author[X. T. Duong]{Xuan Thinh Duong}
\address{Xuan Thinh Duong \\
Department of Mathematics, Macquarie University, NSW 2109, Australia}
\email{xuan.duong@mq.edu.au}

\author[J. Li]{Ji Li}
\address{Ji Li \\
Department of Mathematics, Macquarie University, NSW 2109, Australia}
\email{ji.li@mq.edu.au}

\author[F. K. Ly]{Fu Ken Ly}
\address{Fu Ken Ly \\
Mathematics Learning Centre, University of Sydney, NSW 2006, Australia}

\email{ken.ly@sydney.edu.au}


\arraycolsep=1pt

\begin{abstract}
Let $\La$ be a Schr\"odinger operator with
inverse square potential $a|x|^{-2}$ on $\Rd, d\geq 3$. The main aim of this paper is to prove weighted estimates for fractional powers of $\La$. The proof is based on weighted Hardy inequalities and weighted inequalities for square functions associated to $\La$. As an application, we obtain smoothing estimates regarding the propagator $e^{it\La}$.
\end{abstract}

\maketitle

\section{Introduction}
In this paper, we consider the following Schr\"odinger operators with inverse-square potentials on $\mathbb{R}^d, d\geq 3$,
\begin{equation}\label{defn-La}
\La = -\Delta+\f{a}{|x|^2} \quad \text{with} \quad a\geq -\Big(\f{d-2}{2}\Big)^2.
\end{equation}
Set
$$
\sigma:=\f{d-2}{2}-\f{1}{2}\sqrt{(d-2)^2+4a}.
$$

The Schr\"odinger operator  $\La$ is understood as the Friedrichs extension of $-\Delta+\f{a}{|x|^2}$  defined initially on $C^\vc_c(\Rd\backslash\{0\})$. The condition $a\geq -\bigl(\f{d-2}{2}\bigr)^2$ guarantees that $\La$ is nonnegative.
 It is well-known that $\La$ is self-adjoint and the extension may not be unique as $-\bigl(\f{d-2}{2}\bigr)^2<a<1-\bigl(\f{d-2}{2}\bigr)^2$. For further details, we refer the readers to \cite{K.et.al, K.et.al 2, Ka, PST, T, MZZ}.

It is well known that Schr\"odinger operators with inverse-square potentials $\La$ have a wide range of applications in physics and mathematics spanning areas such as combustion theory, the Dirac equation with Coulomb potential, quantum mechanics and
the study of perturbations of classic space-time metrics. See for example \cite{Bu1, Bu2, VZ, Ka} and their references.

Recently there has been a spate of activity dedicated to the operator $\La$. Strichartz estimates, which are an effective tool for studying the behavior of solutions to nonlinear Schr\"odinger equations and wave equations related to $\La$, were investigated in \cite{Bu1, Bu2} . In \cite{GVV} the authors developed the study of Strichartz estimates for the propagators $e^{it(\Delta+V)}$ with  $V(x)\sim|x|^{-2}$.
The well-posedness and behaviour of the solutions to the heat equation related to $\La$ was studied in \cite{VZ}. In \cite{ZZ}, using  Morawetz-type inequalities and Sobolev norm properties related to $\La$, the long-time behavior of solutions to nonlinear Schr\"odinger equations associated to $\La$ was considered. More recently, the authors in \cite{K.et.al} established the equivalence between  $L^p$-based Sobolev norms  defined
in terms of $\La^{s/2}$ and in terms of $(-\Delta)^{s/2}$ for all regularities $0 < s < 2$.

In this paper, our first objective is to extend the estimates in \cite{K.et.al} to weighted estimates. More precisely, we will prove the following result.

\begin{thm}\label{mainthm}
	Suppose that $d\geq 3$, $a\geq -\left(\f{d-2}{2}\right)^2$ and $0<s<2$. If $r_1:=1\vee \f{d}{d-\sigma}<p<\f{d}{(s+\sigma)\vee 0}:=r_2$ (where $a\vee b=\max\{a,b\}$) with convention $\f{d}{0}=\vc$ then for $w\in A_{p/r_1}\cap RH_{(r_2/p)'}$ we have
	\begin{equation}\label{eq1-mainthm}
	\|(-\Delta)^{s/2}f\|_{L^p_w}\lesi \|\La^{s/2}f\|_{L^p_w}.
	\end{equation}
If	$1<p<\vc$ with $p_1:=1\vee \f{d}{d-\sigma}<p<\f{d}{s\vee \sigma}:=p_2$ then for $w\in A_{p/p_1}\cap RH_{(p_2/p)'}$ we have
\begin{equation}\label{eq2-mainthm}
\|\La^{s/2}f\|_{L^p_w}\lesi \|(-\Delta)^{s/2}f\|_{L^p_w}
\end{equation}
\end{thm}
The proof of the theorem relies havily on the heat kernel of $\La$ in \cite{MS, LS} (see Theorem \ref{thm-heatkernelLa}) which is valid for $d\geq 3$. This is a main reason for the restriction $d\geq 3$.

Let us describe the motivation for the results in Theorem \ref{mainthm}.
\begin{enumerate}[{\rm (i)}]
	\item When $s=1$, \eqref{eq1-mainthm} and \eqref{eq2-mainthm} are known as the boundedness of the Riesz transforms and the reverse Riesz transforms, respectively. Note that the boundedness of the Riesz transforms related to $\La$ was obtained in \cite{HL}. Hence, Theorem \ref{mainthm} can be considered as a natural outgrowth of this direction of research.
	\item The second motivation of our present work is the need of the following estimate of the form:
	$$
	\||x|^\beta \La^\theta f \|_{L^2}\lesi \||x|^\beta(-\Delta)^\theta f\|_{L^2}
	$$
	for certain $\beta$ and $\theta$. This type of estimate was studied in \cite{CD} for certain Schr\"odinger operators instead of $\La$ and played a key role in studying dispersive properties of Schr\"odinger equations on non-flat waveguides. See also \cite{DL} for related weighted estimates in
  $L^{p}$ spaces with mixed radial--angular integrability.
	\item Another motivation of Theorem \ref{mainthm} is its utility in obtaining smoothing estimates related to the propagators $e^{it\La}$. We give such estimates in Theorem \ref{the:mainthm3} below. Note that smoothing estimates related to Schr\"odinger operators are a topic of interest in PDEs and have a close relationship to Strichartz estimates. For further details, the reader can consult \cite{RS, Da, IK, KY, Ma} and the references therein.
	\end{enumerate}
	
	As an application of Theorem \ref{mainthm}, we obtain the following smoothing estimates.


\begin{thm}\label{the:mainthm3}
	Suppose that $d\geq 3$, $a> -\left(\f{d-2}{2}\right)^2+\f{1}{4}$, and 
	consider the Schr\"{o}dinger flow $e^{it\La}f$. Then for all $0<\epsilon<1 $ we have the following smoothing estimates, with $C$ independent of $\epsilon$:
  \begin{equation}\label{eq:firstest}
    \int
    \int
    \left[
    	\frac{|x|^{\epsilon-1}}{(1+|x|^{\epsilon})^{2}}
    	|\nabla e^{it\La}f|^{2}
    	+
    	\frac{ |x|^{\epsilon-3}}{1+|x|^{\epsilon}}
    	|e^{it\La}f|^{2}
    \right]
    dxdt
    \le
    C\epsilon^{-1}\|(-\Delta)^{1/4} f\|_{L^{2}}^{2}
  \end{equation}
  and also
  \begin{equation}\label{eq:secondest}
    \|w(x)^{1/2}
    (-\Delta)^{1/4} e^{it\La}f\|_{L^{2}(\mathbb{R}^{d+1})}
    \lesssim
    C \epsilon^{-1/2}\|f\|_{L^{2}},
    \qquad
    w(x)=\frac{|x|^{\epsilon-1}}{(1+|x|^{\epsilon})^{2}}.
  \end{equation}
  On the other hand, for the wave flow $e^{it \mathcal{L}^{1/2}_{a}}f$
  we have the estimate
  \begin{equation}\label{eq:thirdest}
    \|w(x)^{1/2} e^{it\La^{1/2}}f\|_{L^{2}(\mathbb{R}^{d+1})}
    \lesssim
    C \epsilon^{-1/2}\|f\|_{L^{2}},
    \qquad
    w(x)=\frac{|x|^{\epsilon-1}}{(1+|x|^{\epsilon})^{2}}.
  \end{equation}
\end{thm}
\begin{rem}The condition $a> -\left(\f{d-2}{2}\right)^2+\f{1}{4}$ guarantees that the weight $w$ satisfies the conditions in Theorem \ref{mainthm}. To the best of our knowledge, it is an open question whether the results in Theorem \ref{the:mainthm3} holds true for $a> -\left(\f{d-2}{2}\right)^2$. In the case the dimension $d=3$, the potential $V$ is repulsive and the results of Theorem 1.2 were included in the references \cite{FV,BRV}. Theorem \ref{the:mainthm3} is new for dimensions $d>3$.
\end{rem}

To prove Theorem \ref{mainthm} although we follow the approach in \cite{K.et.al}, some significant modifications and improvements are required due to the following reasons. The first reason is that we work on the weighted Lebesgue estimates instead of unweighted estimates. The second one we need to point out is that in our present paper we employ the vertical square functions in place of the (discrete) Littlewood--Paley square functions. This allows us to bypass the use of spectral multipliers  as in \cite{K.et.al}. Moreover, due to the lack of regularity condition of the heat kernels of $e^{-t\La}$, certain singular integrals considered in the paper may be beyond the Calder\'on-Zygmund theory. This causes some challenging matters, and we overcome these problems by using the criteria established in \cite{AM, BZ} for a singular integrals to be bounded on weighted Lebesgue spaces.

The organization of the paper is as follows. In Section 2 we recall some preliminaries on the Muckenhoupt weights and  two criteria  for a singular integrals to be bounded on weighted and unweighted Lebesgue spaces. Some kernels estimates will be derived in Section 3. In Section 4, we first prove the weighted Hardy inequality and weighted estimates for square functions related to $\La$ which are of interest in their own right. We conclude Section 4 by using these results to prove  Theorems \ref{mainthm} and  \ref{the:mainthm3}.

Throughout the paper, we always use $C$ and $c$ to denote positive constants that are independent of the main parameters involved but whose values may differ from line to line. We will write
$A\lesi B$ if there is a universal constant $C$ so that $A\leq CB$ and $A\sim B$ if $A\lesi B$ and $B\lesi A$. For $a, b\in \mathbb{R}$, we  denote $a\vee b=\max\{a,b\}$ and $a\wedge b=\min\{a,b\}$. For $p\in [1,\vc]$, we denote by $p' =\f{p}{p-1}$ the conjugate exponent of $p$.
\section{Preliminaries}

\subsection{Muckenhoupt weights}
We start with some notations which will be used frequently. For a measurable subset $E\subset \Rd$ and a measurable function $f$ we denote
$$
\fint_E f(x)dx=\f{1}{|E|}\int_E f(x)dx.
$$
Given a ball $B$, we denote $S_j(B)=2^{j}B\backslash 2^{j-1}B$ for $j=1, 2, 3, \ldots$, and we set $S_0(B)=B$.

Let $1\leq q<\infty$. A nonnegative locally integrable function $w$ belongs to the {\sl Muckenhoupt class} $A_q$, say $w\in A_q$, if there exists a positive constant $C$ so that
$$\Big(\fint_B w(x)dx\Big)\Big(\fint_B w^{-1/(q-1)}(x)dx\Big)^{q-1}\leq C, \quad\mbox{if}\; 1<q<\infty,$$
and
$$
\fint_B w(x)dx\leq C \mathop{\mbox{ess-inf}}\limits_{x\in B}w(x),\quad{\rm if}\; q=1,
$$
for all balls $B$ in $\mathbb R^d$. We say that $w\in A_\infty$ if $w\in A_q$ for any $q\in [1,\infty)$. We shall denote $w(E) :=\int_E w(x)dx$ for any measurable set $E \subset \mathbb{R}^d$.

The reverse H\"older classes are defined in the following way: $w\in RH_r, 1 < r < \infty$, if there is a constant $C$ such that for
any ball $B \subset \mathbb{R}^d$,
$$
\Big(\fint_B w^r (x) dx\Big)^{1/r} \leq C \fint_B w(x)dx.
$$
The endpoint $r = \infty$ is given by the condition: $w \in RH_\infty$ whenever, there is a constant $C$ such that for any ball
$B \subset \mathbb{R}^d$,
$$
w(x)\leq C \fint_B w(y)dy  \ \text{for a.e. $x\in B$}.
$$

Let $w \in A_\vc$ . For $0< p <\infty$, the weighted space $L^p_w(\mathbb{R}^d)$ is defined as  the space of $w(x)dx$-measurable functions $f$ such that
$$\|f\|_{L^p_w(\mathbb{R}^d)}:=\Big(\int_{\mathbb{R}^d} |f(x)|^p w(x)dx\Big)^{1/p}<\infty.$$

It is well-known that the power weight $w(x)=|x|^\alpha \in A_p$ if and only if $-d<\alpha<d(p-1)$. Moreover, $w(x)=|x|^\alpha \in RH_q$ if and only if $\alpha q>-d$.

We sum up some of the properties of Muckenhoupt classes and reverse H\"older classes in the following
results. See \cite{Du, JN}.
\begin{lem}\label{weightedlemma1}
	The following properties hold:
	\begin{enumerate}[{\rm (i)}]
		\item $w\in A_p, 1<p<\vc$ if and only if $w^{1-p'}\in A_{p'}$.
		\item $A_1\subset A_p\subset A_q$ for $1\leq p\leq q\leq \infty$.
		\item $RH_\infty \su RH_q \su RH_p$ for $1< p\leq q\leq \infty$.
		\item If $w \in A_p, 1 < p < \vc$, then there exists $1 < q < p$ such that $w \in A_q$.
		\item If $w \in RH_q, 1 < q < \vc$, then there exists $q < p < \infty$ such that $w \in RH_p$.
		\item $A_\vc =\cup_{1\leq p<\vc}A_p = \cup_{1< p\leq \vc}RH_p$.
		\item Let $1<p_0 < p < q_0<\vc$. Then we have
		$$
		w\in A_{\f{p}{p_0}}\cap RH_{(\f{q_0}{p})'}\Longleftrightarrow
		w^{1-p'}\in A_{\f{p'}{q'_0}}\cap RH_{(\f{p'_0}{p'})'}.
		$$
	\end{enumerate}
\end{lem}

\subsection{Hardy-Littlewood maximal functions}

For $r>0$, the Hardy-Littlewood maximal function $\mathcal{M}_r$ is defined by
$$
\mathcal{M}_rf(x)=\sup_{B\ni x}\Big(\f{1}{|B|}\int_B|f(y)|^r\,dy\Big)^{1/r}, \ x\in \mathbb{R}^d,
$$
where the supremum is taken over all balls $B$ containing $x$. When $r=1$, we write $\mathcal{M}$ instead of $\mathcal{M}_1$.

We now record the following results concerning the weak type estimates and the weighted estimates of the maximal functions.
\begin{lem}[\cite{St}]\label{Lem-maximalfunction}
	Let $0< r<\vc$. Then we have for $p>r$ and $w\in A_{p/r}$,
		$$
		\|\mathcal{M}_rf\|_{L^p_w}\lesi \|f\|_{L^p_w}.
		$$
	\end{lem}

\subsection{Two theorems on the boundedness of singular integrals}
We recall the definition of linearizable operators in \cite{GF}. An operator $T$ defined on $L^2(\mathbb{R}^d)$ is said to be a \emph{linearizable operator} if there exists a Banach space $\mathbb{B}$ and a linear operator $U$ from $L^2(\mathbb{R}^d)$ into $L^2(\mathbb{R}^d, \mathbb{B})$ so that
$$
|Tf(x)|=\|Uf(x)\|_{\mathbb{B}}
$$
for all $f\in L^2(\mathbb{R}^d)$ and a.e. $x\in \mathbb{R}^d$.

It can be verified that a linearizable operator is a sublinear operator. The class of linearizable operator includes linear operators, maximal operators and square functions.

We first recall a theorem which is taken from \cite[Theorem 6.6]{BZ} on a criterion for the singular integrals to be bounded on the weighted Lebesgue spaces.
\begin{thm}\label{BZ-thm}
	Let $1\leq p_0< q_0< \vc$ and let $T$ be a linearizable operator. Assume that $T$ can be extended to be bounded on $L^{q_0}$. Assume that there exists a family of operators $\{\mathcal{A}_t\}_{t>0}$ satisfying that for $j\geq 2$ and every ball $B$
	\begin{equation}\label{eq1-BZ}
	\Big(\fint_{S_j(B)}|T(I-\mathcal{A}_{r_B})f|^{q_0}d\mu\Big)^{1/q_0}\leq
	\alpha(j)\Big(\fint_B |f|^{q_0}d\mu\Big)^{1/q_0},
	\end{equation}
	and
	\begin{equation}\label{eq1-BZ-bis}
	\Big(\fint_{S_j(B)}|\mathcal{A}_{r_B}f|^{q_0}d\mu\Big)^{1/q_0}\leq
	\alpha(j)\Big(\fint_B |f|^{p_0}d\mu\Big)^{1/p_0},
	\end{equation}
	for all $f$ supported in $B$. If $\sum_j \alpha(j)2^{jd}<\vc$, then $T$ is bounded on $L^p_w(\mathbb{R}^d)$ for all $p\in (p_0,q_0)$ and $w\in
		A_{\f{p}{p_0}}\cap RH_{(\f{q_0}{p})'}$.
\end{thm}
Note that \cite[Theorem 6.6]{BZ} proves Theorem \ref{BZ-thm} for $q_0=2$, but their arguments also work well for any value $q_0\in (1,\vc)$.

The following theorem is a direct consequence of \cite[Theorem 3.7]{AM} which give a sufficient conditions for a singular integral to be bounded on Lebesgue spaces which plays an important role in the sequel.

 \begin{thm}\label{Martell-thm}

Let $1\leq  p_0< q_0\leq \infty.$  Let $T$ be a bounded sublinear
operator on $L^{p_0}(\mathbb{R}^d)$. Assume that there exists  a family of operators $\{\mathcal{A}_t\}_{t>0}$  satisfying that
\begin{eqnarray}\label{e1-Martell}
\Big( \fint_{B} \big| T(I-\mathcal{A}_{r_B})f\big|^{p_0}dx\Big)^{1/p_0} \leq
C \mathcal{M}_{p_0}(f)(x),
\end{eqnarray}
and
\begin{eqnarray}\label{e2-Martell}
 \Big( \fint_{B} \big| T\mathcal{A}_{r_B}f\big|^{q_0}dx\Big)^{1/q_0} \leq
C \mathcal{M}_{p_0}(|Tf|)(x),
\end{eqnarray}
\noindent for all balls
$B$ with radius $r_B$, all $f \in C^{\infty}_c(\mathbb{R}^d) $ and all $x\in B$. Then $T$ is bounded on $L^p(\mathbb{R}^d)$ for all
 $p_0<p<q_0$.
\end{thm}

\section{Some kernel estimates}

For a constant $\alpha\in \mathbb{R}$. We denote
$$
d_\alpha=\begin{cases}
\f{d}{\alpha}, &\alpha>0\\
\vc, &\alpha\leq 0.
\end{cases}
$$

\begin{thm}\label{thm-Tt}
	Let $\{T_t\}_{t>0}$ be a family of linear operators on $L^2(\Rd)$ with their associated kernels $T_t(x,y)$. Assume that there exist $C,c>0$  and $\alpha, \beta\in \mathbb{R}$ with $d'_\beta<d_\alpha$ such that for all $t>0$ and $x,y \in \Rd\backslash \{0\}$,
	\begin{equation}\label{kernelTt}
	|T_t(x,y)|\leq C\Big(1+\f{\sqrt{t}}{|x|}\Big)^\alpha\Big(1+\f{\sqrt{t}}{|y|}\Big)^\beta t^{-d/2}e^{-\f{|x-y|^2}{ct}}.
	\end{equation}
	Assume that $d'_\beta<p\leq q< d_\alpha$. Then for every $t> 0$, any measurable subsets $E, F\subset \Rd$, and all $f\in L^p(E)$, we have:
	\begin{equation}\label{eq1-Tt}
	\left\|T_tf\right\|_{L^q(F)}\leq Ct^{-\f{d}{2}(\f{1}{p}-\f{1}{q})}e^{-\f{d(E,F)^2}{ct}}\|f\|_{L^p(E)}.
	\end{equation}
\end{thm}

To prove this theorem, we need the following elementary results.
\begin{lem}\label{lem1-Tt}
	(a) Let $\kappa\in (-\vc,d)$. Then there exists $C>0$ so that for all $r>0$
	$$
	\int_{B(0,r)}\f{1}{|x|^{\kappa}}dx\leq Cr^{d-\kappa}.
	$$
	
	(b) For $p\in [1, \vc)$ we have
	$$
	\Big(\int_{\Rd}\Big[\f{1}{t^{d/2}}e^{-\f{|x-y|^2}{ct}}\Big]^pdy\Big)^{1/p}\leq \f{C}{t^{\f{d}{2p'}}}
	$$
	uniformly in $x\in \Rd$.
\end{lem}
\begin{proof}
	The proof of this lemma is simple and we omit details.
\end{proof}

We now turn to prove Theorem \ref{thm-Tt}.

\begin{proof}[Proof of Theorem \ref{thm-Tt}:]

	\noindent Assume \eqref{kernelTt} holds for some $\alpha,\beta\in \mathbb{R}$. Then for $x\in F$ we have
	$$
	\begin{aligned}
	\|T_tf\|_{L^q(F)}&\leq \left\{\int_{F}\Big[\int_E \Big(1+\f{\sqrt{t}}{|x|}\Big)^\alpha\Big(1+\f{\sqrt{t}}{|y|}\Big)^\beta t^{-d/2}e^{-\f{|x-y|^2}{ct}}|f(y)|dy\Big]^qdx\right\}^{1/q}\\
	&\leq \left\{\int_{F\cap B(0,\sqrt{t})}\Big[\int_{E\cap B(0,\sqrt{t})} \Big(1+\f{\sqrt{t}}{|x|}\Big)^\alpha\Big(1+\f{\sqrt{t}}{|y|}\Big)^\beta t^{-d/2}e^{-\f{|x-y|^2}{ct}}|f(y)|dy\Big]^qdx\right\}^{1/q}\\
	&\quad \quad + \left\{\int_{F\cap B(0,\sqrt{t})}\Big[\int_{E\backslash B(0,\sqrt{t})} \Big(1+\f{\sqrt{t}}{|x|}\Big)^\alpha\Big(1+\f{\sqrt{t}}{|y|}\Big)^\beta t^{-d/2}e^{-\f{|x-y|^2}{ct}}|f(y)|dy\Big]^qdx\right\}^{1/q}\\
	&\quad \quad + \left\{\int_{F\backslash B(0,\sqrt{t})}\Big[\int_{E\cap B(0,\sqrt{t})} \Big(1+\f{\sqrt{t}}{|x|}\Big)^\alpha\Big(1+\f{\sqrt{t}}{|y|}\Big)^\beta t^{-d/2}e^{-\f{|x-y|^2}{ct}}|f(y)|dy\Big]^qdx\right\}^{1/q}\\
	&\quad \quad + \left\{\int_{F\backslash B(0,\sqrt{t})}\Big[\int_{E\backslash B(0,\sqrt{t})} \Big(1+\f{\sqrt{t}}{|x|}\Big)^\alpha\Big(1+\f{\sqrt{t}}{|y|}\Big)^\beta t^{-d/2}e^{-\f{|x-y|^2}{ct}}|f(y)|dy\Big]^qdx\right\}^{1/q}\\
	&\leq E_1+E_2+E_3+E_4.
	\end{aligned}
	$$
	By H\"older's inequality, Lemma \ref{lem1-Tt} and the fact that $\beta p'<d$ we have
	$$
	\begin{aligned}
	E_1&\leq Ce^{-\f{d(E,F)^2}{ct}}\Big[\int_{F\cap B(0,\sqrt{t})}t^{-qd/2}\Big(1+\f{\sqrt{t}}{|x|}\Big)^{q\alpha}dx\Big]^{1/q} \Big(\int_E|f|^p\Big)^{1/p}\Big[\int_{B(0,\sqrt{t})} \Big(1+\f{\sqrt{t}}{|y|}\Big)^{\beta p'} dy\Big]^{1/p'}\\
	&\lesi e^{-\f{d(E,F)^2}{ct}}t^{-\f{d}{2}(\f{1}{p}-\f{1}{q})} \Big(\int_E|f|^p\Big)^{1/p}.
	\end{aligned}
	$$
	For the second term, using H\"older's inequality, Lemma \ref{lem1-Tt} again with the fact that $\alpha q<d$ we arrive at
	$$
	\begin{aligned}
	E_2&\lesi \left\{\int_{F\cap B(0,\sqrt{t})}\Big[\int_{E\backslash B(0,\sqrt{t})} \Big(1+\f{\sqrt{t}}{|x|}\Big)^\alpha t^{-d/2}e^{-\f{|x-y|^2}{ct}}|f(y)|dy\Big]^qdx\right\}^{1/q}\\
	&\lesi e^{-\f{d(E,F)^2}{2ct}}\left\{\int_{F\cap B(0,\sqrt{t})}\Big|\Big(1+\f{\sqrt{t}}{|x|}\Big)^\alpha \Big(\int_E|f|^p\Big)^{1/p} \Big(\int_{\Rd}\Big[t^{-d/2}e^{-\f{|x-y|^2}{2ct}}\Big]^{p'}dy\Big)^{1/p'}\Big|^qdx\right\}^{1/q}\\
	&\lesi e^{-\f{d(E,F)^2}{2ct}}t^{-\f{d}{2p}} \Big(\int_E|f|^p\Big)^{1/p}\Big[\int_{F\cap B(0,\sqrt{t})}\Big(1+\f{\sqrt{t}}{|x|}\Big)^{\alpha q}dx\Big]^{1/q}\\
	&\lesi e^{-\f{d(E,F)^2}{2ct}}t^{-\f{d}{2}(\f{1}{p}-\f{1}{q})} \Big(\int_E|f|^p\Big)^{1/p}.
	\end{aligned}
	$$
	By a similar argument we can also dominate $E_3$ by
	$$
	t^{-\f{d}{2}(\f{1}{p}-\f{1}{q})} e^{-\f{d(E,F)^2}{2ct}}\Big(\int_E|f|^p\Big)^{1/p}.
	$$
	It remains to estimate the last term $E_4$. We observe that
	$$
	\begin{aligned}
	E_4\lesi \left\{\int_{F\backslash B(0,\sqrt{t})}\Big[\int_{E\backslash B(0,\sqrt{t})} t^{-d/2}e^{-\f{|x-y|^2}{ct}}|f(y)|dy\Big]^qdx\right\}^{1/q}
	\end{aligned}
	$$
	At this stage, by using the standard argument we can prove that
	$$
	\left\{\int_{F\backslash B(0,\sqrt{t})}\Big[\int_{E\backslash B(0,\sqrt{t})} t^{-d/2}e^{-\f{|x-y|^2}{ct}}|f(y)|dy\Big]^qdx\right\}^{1/q}\lesi t^{-\f{d}{2}(\f{1}{p}-\f{1}{q})} e^{-\f{d(E,F)^2}{2ct}}\Big(\int_E|f|^p\Big)^{1/p}.
	$$
	Hence,
	$$
	E_4\lesi t^{-\f{d}{2}(\f{1}{p}-\f{1}{q})} e^{-\f{d(E,F)^2}{2ct}}\Big(\int_E|f|^p\Big)^{1/p}.
	$$
	This completes the proof of \eqref{eq1-Tt}.
	
\end{proof}

\begin{thm}[\cite{MS, LS}]\label{thm-heatkernelLa}
Assume $d\geq 3$ and $a\geq -\Big(\f{d-2}{2}\Big)^2$. Let $p_t(x, y)$ be the
kernel associated to the semigroups $e^{-t\La}$. Then there exist two positive constants $C$ and $c$ such that for all $t>0$ and $x,y \in \Rd\backslash \{0\}$,
$$
p_t(x,y)\leq C\Big(1+\f{\sqrt{t}}{|x|}\Big)^\sigma\Big(1+\f{\sqrt{t}}{|y|}\Big)^\sigma t^{-d/2}e^{-\f{|x-y|^2}{ct}}.
$$
\end{thm}

The following results gives some estimates of the heat kernels $p_z(x,y)$ for $z\in \mathbb{C}_{\pi/4}:=\{z\in \mathbb{C}: |\arg z|<\pi/4\}$.
\begin{prop}\label{heatkernelestimates-halfplane}
	Let $p_z(x,y)$ be the kernels associated to the semigroups $e^{-z\La}$ with $z\in \mathbb{C}_{\pi/4}:=\{z\in \mathbb{C}: |\arg z|<\pi/4\}$. Then there exists constants $C$ and $c$ such that
	\begin{equation}\label{boundedp_t(x,y)-complex}
	|p_z(x,y)|\leq C\Big(1+\f{\sqrt{|z|}}{|x|}\Big)^\sigma\Big(1+\f{\sqrt{|z|}}{|y|}\Big)^\sigma |z|^{-d/2} e^{-\f{|x-y|^2}{c|z|}}
	\end{equation}
\end{prop}

\begin{proof} We adapt the standard argument in \cite{Da} to our present situation.

It suffices to claim that
\begin{equation}\label{boundedp_t(x,y)-complex1'}
\left|w(x) p_z(x,y)w(y)\right|\leq \f{C}{|z|^{d/2}},
\end{equation}
where $w(x)=\Big(1+\f{\sqrt{|z|}}{|y|}\Big)^{-\sigma}$

Now for $f: \mathbb{R}^d\rightarrow \mathbb{R}$, we define the norm
$$
|f|_{wL^\vc}=\sup_{x} \left|f(x)w(x)\right|.
$$
Hence (\ref{boundedp_t(x,y)-complex1'}) is equivalent to that
$$
\|e^{-z\La}\|_{L^1_{w^{-1}}\rightarrow wL^\vc}\leq \f{C}{|z|^{d/2}}.
$$
Assume that $z=2t+is$ where $t\geq 0$ and $s\in \mathbb{R}$. Due to $|\arg z|<\pi/4$, we have $t\sim |z|$. Hence
$$
\|e^{-z\La}\|_{L^1_{w^{-1}}\rightarrow wL^\vc}\leq \|e^{-t\La}\|_{L^2\rightarrow wL^\vc}\|e^{-is\La}\|_{L^2\rightarrow L^2}\|e^{-t\La}\|_{L^1_{w^{-1}}\rightarrow L^2}.
$$
Since $\La$ is nonnegative and self-adjoint,  $\|e^{-is\La}\|_{L^2\rightarrow L^2}\leq 1$.  We now claim that
$$
 \|e^{-t\La}\|_{L^1_{w^{-1}}\rightarrow L^2} \lesi t^{-d/4}\ \ \ \text{and} \ \ \ \ \|e^{-t\La}\|_{L^2\rightarrow wL^\vc}\lesi t^{-d/4}.
$$
We now show $\|e^{-t\La}\|_{L^1_{w^{-1}}\rightarrow L^2} \lesi t^{-d/4}$. The inequality
$\|e^{-t\La}\|_{L^2\rightarrow wL^\vc}\lesi t^{-d/4}$ can be done in the same manner. Indeed, for $f\in L^1_{w^{-1}}$ we have
$$
\begin{aligned}
\|e^{-t\La}f\|_{L^2}&\lesi \left(\int_{\mathbb{R}^d}\left|\int_{\mathbb{R}^d} \Big(1+\f{\sqrt{t}}{|x|}\Big)^\sigma\Big(1+\f{\sqrt{t}}{|y|}\Big)^\sigma t^{-d/2}e^{-\f{|x-y|^2}{ct}}|f(y)|dy\right|^2dx\right)^{1/2}\\
&\lesi \int_{\mathbb{R}^d}\left(\int_{\mathbb{R}}\left| \Big(1+\f{\sqrt{t}}{|x|}\Big)^\sigma t^{-d/2}e^{-\f{|x-y|^2}{ct}}\right|^2dx\right)^{1/2}\Big(1+\f{\sqrt{t}}{|y|}\Big)^\sigma|f(y)|dy
\end{aligned}
$$
Arguing similarly to the proof of Theorem \ref{thm-Tt} we get that
$$
\left(\int_{\mathbb{R}^d}\left| \Big(1+\f{\sqrt{t}}{|x|}\Big)^\sigma t^{-d/2}e^{-\f{|x-y|^2}{ct}}\right|^2dx\right)^{1/2}\lesi t^{-d/4}.
$$
Hence,
$$
\|e^{-t\La}f\|_{L^2}\lesi t^{-d/4}\int_{\mathbb{R}^d}\Big(1+\f{\sqrt{t}}{|y|}\Big)^\sigma|f(y)|dy:=t^{-d/4}\|f\|_{L^1_{w^{-1}}}
$$
which implies
$$
\|e^{-t\La}\|_{L^1_{w^{-1}}\rightarrow L^2} \lesi t^{-d/4}.
$$
This completes our proof.

\end{proof}

As a direct consequence of Proposition \ref{heatkernelestimates-halfplane} and Cauchy formula, we obtain the following result.
\begin{prop}\label{thm-ptk}
Assume $d\geq 3$ and $a\geq -\Big(\f{d-2}{2}\Big)^2$. For any $k\in \mathbb{N}$, there exist two positive constants $C_k$ and $c_k$ such that for all $t>0$ and $x,y \in \Rd\backslash \{0\}$,
$$
|p_{t,k}(x,y)|\leq C_k\Big(1+\f{\sqrt{t}}{|x|}\Big)^\sigma\Big(1+\f{\sqrt{t}}{|y|}\Big)^\sigma t^{-(k+d/2)}e^{-\f{|x-y|^2}{c_kt}},
$$
where $p_{t,k}(x,y)$ is an associated kernel to $\La^k e^{-t\La}$.
\end{prop}
\section{Riesz transforms and smoothing estimates}
\subsection{Weighted Hardy inequalities}
The Hardy inequality for Laplacian $-\Delta$ was studied in \cite{T} and then was generalized for Schr\"odinger operators $\La$ in \cite{K.et.al}. In this section, we extend to the weighted Hardy inequalities for  $\La$.
\begin{thm}\label{thm-HardyIneq}
Suppose $0 <s<d, d-s-2\sigma >
0$, and $d_\sigma'<p<d_{s+\sigma}$. Then for $w\in A_{\f{p}{d'_\sigma}}\cap RH_{\left(\f{d_{s+\sigma}}{p}\right)'}$ we have
$$
\||x|^{-s}f\|_{L^p_w(\Rd)}\lesi \|\La^{s/2}f\|_{L^p_w(\Rd)}.
$$
\end{thm}
\begin{proof}
It suffices to prove that
$$
\||x|^{-s}\La^{-s/2}g\|_{L^p_w(\Rd)}\lesi \|g\|_{L^p_w(\Rd)}.
$$
for $d_\sigma'<p<d_{s+\sigma}$ and $w\in A_{\f{p}{d'_{\sigma}}}\cap RH_{\left(\f{d_{s+\sigma}}{p}\right)'}$. To do so we shall apply Theorem \ref{BZ-thm}.

We define a linear operator
$$
T_{\La, s}f(x)=|x|^{-s}\La^{-s/2}f(x).
$$
Fix $p\in (d'_{\sigma},d_{s+\sigma})$ and $w\in A_{\f{p}{d'_{\sigma}}}\cap RH_{\left(\f{d_{s+\sigma}}{p}\right)'}$. Then we can find $d'_{\sigma}<p_1<p<p_2<d_{s+\sigma}$ so that
$w\in A_{\f{p}{p_1}}\cap RH_{\left(\f{p_2}{p}\right)'}$.

We now fix a ball $B\subset \Rd$ and $m>d+d/p_2$. For any function $f$ supported in $B$ we claim that
\begin{equation}\label{eq1-thmHardyIneq}
\Big(\int_{S_j(B)}|T_{\La, s}(I-e^{-r_B^2\La})^mf(x)|^{p_2}dx\Big)^{1/p_2}\lesi 2^{-2mj}\Big(\int_B|f(x)|^{p_2}dx\Big)^{1/p_2}.
\end{equation}
Indeed, using the formula
$$
\La^{-s/2}=\f{1}{\Gamma(s/2)}\int_0^\vc t^{s/2}e^{-t\La}\f{dt}{t}
$$
to obtain that
\begin{equation}\label{eq2-thmHardyIneq}
\begin{aligned}
T_{\La, s}(I-e^{-r_B^2\La})^mf(x)&=\f{1}{\Gamma(s/2)}\int_0^\vc t^{s/2}|x|^{-s}e^{-t\La}(I-e^{-r_B^2\La})^mf(x)\f{dt}{t}.
\end{aligned}
\end{equation}
This along with Minkowski's inequality implies that
$$
\begin{aligned}
\|T_{\La, s}(I-e^{-r_B^2\La})^mf\|_{L^{p_2}(S_j(B))}&\lesi \f{1}{\Gamma(s/2)}\int_0^\vc t^{s/2}\|\,|x|^{-s}e^{-t\La}(I-e^{-r_B^2\La})^mf(x)\|_{L^{p_2}(S_j(B))}\f{dt}{t}\\
&\lesi \f{1}{\Gamma(s/2)}\int_0^{r_B^2} t^{s/2}\|\,|x|^{-s}e^{-t\La}(I-e^{-r_B^2\La})^mf(x)\|_{L^{p_2}(S_j(B))}\f{dt}{t}\\
&\quad \quad +\f{1}{\Gamma(s/2)}\int_{r_B^2}^\vc t^{s/2}\|\,|x|^{-s}e^{-t\La}(I-e^{-r_B^2\La})^mf(x)\|_{L^{p_2}(S_j(B))}\f{dt}{t}\\
&\lesi E_1 + E_2.
\end{aligned}
$$
We first take care of $E_1$. Observe that
\begin{equation}\label{eq-formulaE1}
E_1\leq \sum_{k=0}^m C^m_k\int_0^{r_B^2} t^{s/2}\|\,|x|^{-s}e^{-(t+kr_B^2)\La}f(x)\|_{L^{p_2}(S_j(B))}\f{dt}{t}.
\end{equation}
Note that the associated kernel of the linear operator $f(x) \mapsto |x|^{-s}e^{-(t+kr_B^2)\La}f(x)$ is given by
$ |x|^{-s}p_{t+kr_B^2}(x,y)$ and by Theorem \ref{thm-heatkernelLa} it is dominated by
$$
\begin{aligned}
|x|^{-s}&\Big(1+\f{\sqrt{t+kr_B^2}}{|x|}\Big)^\sigma\Big(1+\f{\sqrt{t+kr_B^2}}{|y|}\Big)^\sigma (t+kr_B^2)^{-d/2}e^{-\f{|x-y|^2}{c(t+kr_B^2)}}\\
&\leq (t+kr_B^2)^{-s/2}\Big(1+\f{\sqrt{t+kr_B^2}}{|x|}\Big)^{\sigma+s}\Big(1+\f{\sqrt{t+kr_B^2}}{|y|}\Big)^\sigma (t+kr_B^2)^{-d/2}e^{-\f{|x-y|^2}{c(t+kr_B^2)}}.
\end{aligned}
$$
Therefore, applying Theorem \ref{thm-Tt} we get that
$$
\|\,|x|^{-s}e^{-(t+kr_B^2)\La}f(x)\|_{L^{p_2}(S_j(B))}\lesi (t+kr_B^2)^{-s/2}e^{-\f{4^jr_B^2}{c(t+kr_B^2)}}\Big(\int_B|f|^{p_2}\Big)^{1/p_2}.
$$
Inserting this into \eqref{eq-formulaE1} to obtain that
$$
\begin{aligned}
E_1&\lesi \sum_{k=0}^m \int_0^{r_B^2} t^{s/2}(t+kr_B^2)^{-s/2}e^{-\f{4^jr_B^2}{c(t+kr_B^2)}}\f{dt}{t}\times \Big(\int_B|f|^{p_2}\Big)^{1/p_2}\\
&\lesi \sum_{k=0}^m \int_0^{r_B^2} t^{s/2}(t+kr_B^2)^{-s/2}\Big(\f{t+kr_B^2}{4^jr_B^2}\Big)^m \f{dt}{t}\times \Big(\int_B|f|^{p_2}\Big)^{1/p_2}\\
&\lesi 2^{-2mj}\Big(\int_B|f|^{p_2}\Big)^{1/p_2}.
\end{aligned}
$$
To estimate $E_2$, we note that
\begin{equation}\label{eq2-squarefunction}
(I-e^{-r_B^2\La})^m=\int_0^{r_B^2}\dots \int_0^{r_B^2} \La^me^{-(s_1+\dots+s_m)\La}d\vec{s},
\end{equation}
where $d\vec{s}=ds_1\dots ds_m$.

Hence, the associated kernel to the linear operator $f(x)\mapsto |x|^{-s}e^{-t\La}(I-e^{-r_B^2\La})^mf(x)$ is given by
$$
\int_0^{r_B^2}\dots \int_0^{r_B^2} |x|^{-s} p_{t+s_1+\dots+s_m, m} (x,y)d\vec{s},
$$
and hence by using Theorem \ref{thm-ptk} we can dominate it by
$$
\begin{aligned}
|x|^{-s}t^{-m}&\Big(1+\f{\sqrt{t+s_1+\dots +s_m}}{|x|}\Big)^\sigma\Big(1+\f{\sqrt{t+s_1+\dots +s_m}}{|y|}\Big)^\sigma (t+s_1+\dots +s_m)^{-d/2}e^{-\f{|x-y|^2}{c(t+s_1+\dots +s_m)}}\\
&\lesi |x|^{-s}t^{-m}\Big(1+\f{\sqrt{t}}{|x|}\Big)^\sigma\Big(1+\f{\sqrt{t}}{|y|}\Big)^\sigma t^{-d/2}e^{-\f{|x-y|^2}{c(t)}}\\
&\lesi t^{-(s/2+m)}\Big(1+\f{\sqrt{t}}{|x|}\Big)^{\sigma+s}\Big(1+\f{\sqrt{t}}{|y|}\Big)^\sigma t^{-d/2}e^{-\f{|x-y|^2}{c(t)}},
\end{aligned}
$$
where in the first inequality we used the fact that $t+s_1+\dots +s_m\thicksim t$ for $t\geq r_B^2$ and $s_i\in (0,r_B^2], i=1,\dots, m.$

This in combination with Theorem \ref{thm-Tt} implies that
$$
\|\,|x|^{-s}e^{-t\La}(I-e^{-r_B^2\La})^mf(x)\|_{L^{p_2}(S_j(B))}\lesi t^{-(s/2+m)}e^{-\f{4^jr_B^2}{ct}}\Big(\int_B|f|^{p_2}\Big)^{1/p_2}
$$
Inserting this into the expression of $E_2$ to get that
$$
\begin{aligned}
E_2&\lesi \int_{r_B^2}^\vc t^{s/2} t^{-(s/2+m)}e^{-\f{4^jr_B^2}{ct}}\f{dt}{t}\times \Big(\int_B|f|^{p_2}\Big)^{1/p_2}\\
&\lesi \int_{r_B^2}^\vc t^{s/2} t^{-(s/2+m)}\Big(\f{t}{4^jr_B^2}\Big)^m\f{dt}{t}\times \Big(\int_B|f|^{p_2}\Big)^{1/p_2}\\
&\lesi 2^{-2mj}\Big(\int_B|f|^{p_2}\Big)^{1/p_2}.
\end{aligned}
$$
From the estimates of $E_1$ and $E_2$ we conclude \eqref{eq1-thmHardyIneq}.

With estimate \eqref{eq1-thmHardyIneq} in hand, we can now complete the proof of the theorem. First note that it was proved in \cite{K.et.al} that $T_{\La, s}$ is bounded on $L^p$ for  $p\in (d'_{\sigma},d_{s+\sigma})$. For each ball $B$, we now set
$$
\mathcal{A}_{r_B}=I-(I-e^{-r_B^2\La})^m.
$$
Then from \eqref{eq1-thmHardyIneq}, we conclude that
	`\begin{equation}\label{eq1-Hardy}
	\Big(\fint_{S_j(B)}|T_{\La, s}(I-\mathcal{A}_{r_B})f(x)|^{p_2}dx\Big)^{1/p_2}\lesi 2^{-j(2m-d/p_2)}\Big(\fint_B|f(x)|^{p_2}dx\Big)^{1/p_2}.
	\end{equation}
On the other hand from Theorem \ref{thm-Tt} and Proposition \ref{thm-heatkernelLa} we imply that
`\begin{equation}
\Big(\int_{S_j(B)}|\mathcal{A}_{r_B}f(x)|^{p_2}dx\Big)^{1/p_2}\lesi e^{-c2^{2j}}|B|^{\f{1}{p_1}-\f{1}{p_2}} \Big(\int_B|f(x)|^{p_1}dx\Big)^{1/p_1}
\end{equation}
which implies that
`\begin{equation}\label{eq2-Hardy}
\Big(\fint_{S_j(B)}|\mathcal{A}_{r_B}f(x)|^{p_2}dx\Big)^{1/p_2}\lesi 2^{-j(2m-d/p_2)}\Big(\fint_B|f(x)|^{p_1}dx\Big)^{1/p_1}.
\end{equation}
From \eqref{eq1-Hardy}, \eqref{eq2-Hardy} and Theorem \ref{BZ-thm} we obtained the desired result.
\end{proof}

\subsection{Weighted estimates for square functions}

Let $\alpha\in (0,1)$ we consider the following square function
$$
S_{\La,\alpha}f(x)=\Big(\int_0^\vc|(t\La)^{1-\alpha}e^{-t\La}f|^2\f{dt}{t}\Big)^{1/2}.
$$
Note that by functional calculus theory in \cite{Mc}, the square function  $S_{\La,\alpha}$ is bounded on $L^2$. In the following theorem, we prove the weighted $L^p$ estimates for $S_{\La,\alpha}$.
\begin{thm}\label{thm-square}
Suppose that $d\geq 3$, $a\geq -\left(\f{d-2}{2}\right)^2$ and  $\alpha\in (0,1)$. Then for all $d_\sigma'<p<d_\sigma$ and $w\in A_{\f{p}{d_\sigma'}}\cap RH_{(\f{d_\sigma}{p})'}$ we have
$$
\|S_{\La,\alpha}f\|_{L^p_w}\sim\|f\|_{L^p_w}.
$$
As a consequence, for $0<s<2$, $d_\sigma'<p<d_\sigma$ and $w\in A_{\f{p}{d_\sigma'}}\cap RH_{(\f{d_\sigma}{p})'}$ we have
$$
\left\|\Big(\int_0^\vc t^{-s}|t\La e^{-t\La}f|^2\f{dt}{t}\Big)^{1/2}\right\|_{L^p_w}\sim \|\La^{s/2} f\|_{L^p_w}.
$$
\end{thm}
\begin{proof}

We shall apply again Theorem \ref{BZ-thm}. To do so we assume for now that $S_{\La,\alpha}$ is bounded on $L^p$ for $p\in [2,d_\sigma)$. This assumption will be justified later.

Fix $d_\sigma'<p<d_\sigma$ and $w\in A_{\f{p}{d_\sigma'}}\cap RH_{(\f{d_\sigma}{p})'}$. Then we can pick $d_\sigma'<p_1<p< 2<p_2<d_\sigma$ so that $w\in A_{\f{p}{p_1}}\cap RH_{(\f{p_2}{p})'}$.

Fix $m>1 +d+d/p_2$, we will claim that
\begin{equation}\label{eq1-squarefunction}
\Big(\int_{S_j(B)}|S_{\La,\alpha}(I-e^{-r_B^2\La})^mf|^{p_2}dx\Big)^{1/p_2}\lesi 2^{-2(m-1)j}\Big(\int_{B}|f|^{p_2}dx\Big)^{1/p_2},\quad \quad j \geq 2,
\end{equation}
for all balls $B$ and all $f\in C^\vc$ supported in $B$.

Indeed, by Minkowski's inequality we have
$$
\begin{aligned}
\Big(\int_{S_j(B)}|S_{\La,\alpha}(I-e^{-r_B^2\La})^mf|^{p_2}dx\Big)^{1/p_2}&\leq \Big(\int_0^{r_B^2}\left\|(t\La)^{1-\alpha}e^{-t\La}(I-e^{-r_B^2\La})^mf\right\|_{L^{p_2}(S_j(B))}^2\f{dt}{t}\Big)^{1/2}\\
& \quad+\Big(\int_{r_B^2}^\vc\left\|(t\La)^{1-\alpha}e^{-t\La}(I-e^{-r_B^2\La})^mf\right\|_{L^{p_2}(S_j(B))}^2\f{dt}{t}\Big)^{1/2}\\
&:=I_1+I_2.
\end{aligned}
$$
We now take care of $I_1$ first. Note that
$$
\La^{-\alpha}=\f{1}{\Gamma(\alpha)}\int_{0}^\vc u^\alpha e^{-u\La}\f{du}{u}.
$$

Hence,
$$
\begin{aligned}
I_1&\lesi \left(\int_0^{r_B^2}\left[\int_{0}^{r_B^2}\Big(\f{u}{t}\Big)^\alpha\left\|t\La e^{-(t+u)\La}(I-e^{-r_B^2\La})^mf\right\|_{L^{p_2}(S_j(B))}\f{du}{u}\right]^2\f{dt}{t}\right)^{1/2}\\
& \ \ + \left(\int_0^{r_B^2}\left[\int_{r_B^2}^\vc\Big(\f{u}{t}\Big)^\alpha\left\|t\La e^{-(t+u)\La}(I-e^{-r_B^2\La})^mf\right\|_{L^{p_2}(S_j(B))}\f{du}{u}\right]^2\f{dt}{t}\right)^{1/2}\\
&:=I_{11}+I_{12}.
\end{aligned}
$$
We now have
$$
\begin{aligned}
I_{11}&\lesi \sum_{k=0}^m\left(\int_0^{r_B^2}\left[\int_{0}^{r_B^2}\Big(\f{u}{t}\Big)^\alpha\f{t}{t+u+kr_B^2}\left\|(t+u+kr_B^2)\La e^{-(t+u+kr_B^2)\La}f\right\|_{L^{p_2}(S_j(B))}\f{du}{u}\right]^2\f{dt}{t}\right)^{1/2}\\
&\lesi \sum_{k=0}^m\left(\int_0^{r_B^2}\left[\int_{0}^{r_B^2}\Big(\f{u}{t}\Big)^\alpha\f{t}{t+u+kr_B^2}\exp\Big(-\f{4^jr^2_B}{c(t+u+kr_B^2)}\Big)\left\|f\right\|_{L^{p_2}(B)}\f{du}{u}\right]^2\f{dt}{t}\right)^{1/2}\\
&\lesi \sum_{k=0}^m\left(\int_0^{r_B^2}\left[\int_{0}^{r_B^2}\Big(\f{u}{t}\Big)^\alpha\f{t}{4^jr_B^2}\Big(\f{t+u+kr_B^2}{4^jr^2_B}\Big)^{m}\left\|f\right\|_{L^{p_2}(B)}\f{du}{u}\right]^2\f{dt}{t}\right)^{1/2}\\
&\lesi 2^{-2mj}\left\|f\right\|_{L^{p_2}(B)}.
\end{aligned}
$$
Now we use \eqref{eq2-squarefunction} to obtain that
$$
\begin{aligned}
I_{12}&\lesi \left(\int_0^{r_B^2}\left[\int_{[0,r_B^2]^m}\int^\vc_{r_B^2}\Big(\f{u}{t}\Big)^\alpha\left\|t\La^{m+1} e^{-(t+u+s_1+\ldots+s_m)\La}f\right\|_{L^{p_2}(S_j(B))}\f{du}{u}d\vec{s}\right]^2\f{dt}{t}\right)^{1/2}\\
\end{aligned}
$$
which along with Theorem \ref{thm-Tt} and the fact that $u\sim t+u+s_1+\ldots+s_m$ implies that
$$
\begin{aligned}
I_{12}&\lesi \left(\int^{r_B^2}_0\left[\int_{[0,r_B^2]^m}\int^\vc_{r_B^2}\Big(\f{u}{t}\Big)^\alpha \f{t}{u^{m+1}}e^{-\f{4^jr_B^2}{cu}}\left\|f\right\|_{L^{p_2}(B)}\f{du}{u}d\vec{s}\right]^2\f{dt}{t}\right)^{1/2}\\
&\lesi 2^{-2mj} \left\|f\right\|_{L^{p_2}(B)}.
\end{aligned}
$$
As a consequence,
$$
I_1\lesi 2^{-2mj}\left\|f\right\|_{L^{p_2}(B)}.
$$

Similarly, we split $I_2$ as follows
$$
\begin{aligned}
I_2&\lesi \left(\int^\vc_{r_B^2}\left[\int_{0}^{r_B^2}\Big(\f{u}{t}\Big)^\alpha\left\|t\La e^{-(t+u)\La}(I-e^{-r_B^2\La})^mf\right\|_{L^{p_2}(S_j(B))}\f{du}{u}\right]^2\f{dt}{t}\right)^{1/2}\\
& \ \ + \left(\int^\vc_{r_B^2}\left[\int_{r_B^2}^\vc\Big(\f{u}{t}\Big)^\alpha\left\|t\La e^{-(t+u)\La}(I-e^{-r_B^2\La})^mf\right\|_{L^{p_2}(S_j(B))}\f{du}{u}\right]^2\f{dt}{t}\right)^{1/2}\\
&:=I_{21}+I_{22}.
\end{aligned}
$$
The argument used to estimate $I_{12}$ can be applied again to show that
$$
I_{21}\lesi 2^{-2mj}\left\|f\right\|_{L^{p_2}(B)}.
$$
On the other hand, using this argument, we  also dominate $I_{22}$ as follows
$$
\begin{aligned}
I_{22}&\lesi \left(\int_{r_B^2}^\vc\left[\int_{[0,r_B^2]^m}\int^\vc_{r_B^2}\Big(\f{u}{t}\Big)^\alpha \f{t}{(u+t)^{m+1}}e^{-\f{4^jr_B^2}{c(u+t)}}\left\|f\right\|_{L^{p_2}(B)}\f{du}{u}d\vec{s}\right]^2\f{dt}{t}\right)^{1/2}\\
&\lesi \left(\int_{r_B^2}^\vc\left[\int_{[0,r_B^2]^m}\int^\vc_{r_B^2}\Big(\f{u}{t}\Big)^\alpha \f{1}{u(u+t)^{m-1}} e^{-\f{4^jr_B^2}{c(u+t)}}\left\|f\right\|_{L^{p_2}(B)}\f{du}{u}d\vec{s}\right]^2\f{dt}{t}\right)^{1/2}.
\end{aligned}
$$
Using the inequality 
\[
\f{1}{(u+t)^{m-1}}e^{-\f{4^jr_B^2}{c(u+t)}}\lesi 2^{-2(m-1)j}r_B^{-2(m-1)},
\]
and by a simple calculation we obtain
$$
\begin{aligned}
I_{22}&\lesi 2^{-2(m-1)j}\left\|f\right\|_{L^{p_2}(B)}.
\end{aligned}
$$
Therefore,
$$
I_{22}\lesi 2^{-2(m-1)j}\left\|f\right\|_{L^{p_2}(B)}.
$$
Hence, this completes the proof of \eqref{eq1-squarefunction}. At this stage, arguing similarly to Theorem \ref{thm-HardyIneq}, we obtain that
$$
\|S_{\La,\alpha}f\|_{L^p_w}\lesi \|f\|_{L^p_w}.
$$
To prove the reverse inequality, by functional calculus theory for $g\in L^{p'}_v$ with $v=w^{1-p'}$ we have
$$
\begin{aligned}
\int_{\Rd} f(x)g(x)dx&=c(\alpha)\int_{\Rd} \int_0^\vc(t\La)^{2(1-\alpha)}e^{-2t\La}f(x)g(x)\f{dt}{t}dx,
\end{aligned}
$$
where $c(\alpha)= \int_0^\vc t^{2(1-\alpha)}e^{-2t}\f{dt}{t}$. (Actually, this identity holds true in $L^2$ first. However, due to the weighted $L^p$ boundeness of $S_{\La,\alpha}$ we can extend the convergence to $L^p_w$.)

By H\"older's inequality, we can write
$$
\begin{aligned}
\int_{\Rd} f(x)g(x)dx
&=c(\alpha)\int_{\Rd} \int_0^\vc(t\La)^{1-\alpha}e^{-t\La}f(x)(t\La)^{1-\alpha}e^{-t\La}g(x)\f{dt}{t}dx\\
&\lesi \int_{\Rd}S_{\La,\alpha}f(x)S_{\La,\alpha}g(x) dx\\
&\lesi\|S_{\La,\alpha}f\|_{L^p_w}\|S_{\La,\alpha}g\|_{L^{p'}_v}.
\end{aligned}
$$
Note that from (vii) Lemma \ref{weightedlemma1} we obtain $v\in A_{p'/d_\sigma\cap RH_{(d_\sigma'/p')'}}$. Hence, from the weighted $L^p$ estimates of $S_{\La,\alpha}$ we have proved we get that $\|S_{\La,\alpha}g\|_{L^{p'}_v}\lesi \|g\|_{L^{p'}_v}$, we obtain
$$
\int_{\Rd} f(x)g(x)dx\lesi \|S_{\La,\alpha}f\|_{L^p_w}\|g\|_{L^{p'}_v}.
$$
As a consequence,
$$
\|f\|_{L^p_w}\lesi \|S_{\La,\alpha}f\|_{L^p_w}.
$$
\bigskip

To complete the proof, we need to prove the original assertion that $S_{\La,\alpha}$ is bounded on $L^r$ for all $r\in (2, d_\sigma)$. According to Theorem \ref{Martell-thm}, for any $q_0\in (2,d_\sigma)$ it suffices to prove that
\begin{eqnarray}\label{e1-SLa}
\Big( \fint_{B} \left| S_{\La,\alpha}(I-\mathcal{A}_{r_B})f\right|^{2}dx\Big)^{1/2} \leq
C \mathcal{M}_{2}(f)(x),
\end{eqnarray}
and
\begin{eqnarray}\label{e2-SLa}
\Big( \fint_{B} \big| S_{\La,\alpha}\mathcal{A}_{r_B}f\big|^{q_0}dx\Big)^{1/q_0} \leq
C \mathcal{M}_{2}(|S_{\La,\alpha}f|)(x),
\end{eqnarray}
all balls $B$ with radius $r_B$, all $f \in C^{\infty}_c(\mathbb{R}^d) $ and all $x\in B$ with $\mathcal{A}_{r_B}=I-(I-e^{-r_B^2\La})^m$, $m>1+d+d/p_2$.

To prove \eqref{e1-SLa}, we write
$$
\begin{aligned}
\Big( \fint_{B} \left| S_{\La,\alpha}(I-\mathcal{A}_{r_B})f\right|^{2}dx\Big)^{1/2}&\leq \sum_{j=0}^\vc\Big( \fint_{B} \left| S_{\La,\alpha}(I-\mathcal{A}_{r_B})f_j\right|^{2}dx\Big)^{1/2}\\
&:=\sum_{j=0}^\vc I_j,
\end{aligned}
$$
where $f_j=f\chi_{S_j(B)}$.

For $j=0,1$, using the $L^2$-boundedness of $S_{\La,\alpha}$ and $\mathcal{A}_{r_B}$ we have
$$
I_j\lesi \mathcal{M}_{2}(f)(x).
$$
For $j\geq 2$, the argument in the proof of \eqref{eq1-squarefunction} shows that
$$
I_j\lesi 2^{-j(2m-d/2)}\Big(\fint_{S_j(B)}|f|^{2}\Big)^{1/2}.
$$
Therefore,
$$
\sum_{j=0}^\vc I_j\lesi \mathcal{M}_{2}(f)(x)
$$
which proves \eqref{e1-SLa}.

It remains to prove \eqref{e2-SLa}. Indeed, we have
$$
\begin{aligned}
\Big(\int_B &|S_{\La,\alpha}[I-(I-e^{-r_B^2\La})^m]f(x)|^{q_0}dx\Big)^{1/q_0}\\
&\lesi \sum_{k=1}^m \Big(\int_B |S_{\La,\alpha}e^{-kr_B^2\La}f(x)|^{q_0}dx\Big)^{1/q_0}\\
&\lesi \sup_{1\leq k\leq m}\left[\int_B \left( \int_0^\vc |e^{-kr_B^2\La}(t\La)^{1-\alpha}e^{-t\La}f(x)|^2\f{dt}{t} \right)^{q_0/2}dx\right]^{1/q_0}\\
&\lesi\sum_{j\geq 0}\sup_{1\leq k\leq m}\left[\int_B \left( \int_0^\vc \left|e^{-kr_B^2\La}[(t\La)^{1-\alpha}e^{-t\La}f\chi_{S_j(B)}](x)\right|^2\f{dt}{t} \right)^{q_0/2}dx\right]^{1/q_0}
\end{aligned}
$$
which along with Minkowski's inequality, Theorem \ref{thm-ptk} and Theorem \ref{thm-Tt} gives
$$
\begin{aligned}
\Big(\int_B &|S_{\La,\alpha}[I-(I-e^{-r_B^2\La})^m]f(x)|^{q_0}dx\Big)^{1/q_0}\\
&\lesi\sum_{j\geq 0}\sup_{1\leq k\leq m}\left( \int_0^\vc \left\|e^{-kr_B^2\La}[(t\La)^{1-\alpha}e^{-t\La}f\chi_{S_j(B)}]\right\|_{L^{q_0}(B)}^2\f{dt}{t} \right)^{1/2}\\
&\lesi\sum_{j\geq 0} e^{-c4^j}|B|^{-(\f{1}{2}-\f{1}{q_0})}\left( \int_0^\vc \left\|(t\La)^{1-\alpha}e^{-t\La}f\right\|_{L^2(S_j(B))}^2\f{dt}{t} \right)^{1/2}\\
&\lesi\sum_{j\geq 0} e^{-c4^j}|B|^{-(\f{1}{2}-\f{1}{q_0})}\Big(\int_{2^jB} |S_{\La,\alpha}f(x)|^{2}dx\Big)^{1/2}.
\end{aligned}
$$
This implies \eqref{e2-SLa}. Hence the proof is complete.
\end{proof}

The following result regarding weighted estimates for the difference of square functions will play an essential role in the proofs of the main results.
\begin{thm}\label{thm-difference}
	We have the following estimate
$$
\left\|\left(\int_0^\vc t^{-s}\left|\left(t\La e^{-t\La} +t\Delta e^{t\Delta}\right)f\right|^2\f{dt}{t}\right)^{1/2}\right\|_{L^p_w}\lesi \left\|\f{f}{|x|^s}\right\|_{L^p_w}
$$
provided that
\begin{enumerate}[{\rm (a)}]
	\item $a\geq 0$, $1<p<\vc$ and $w\in A_p$; or
	\item $-\left(\f{d-2}{2}\right)^2\leq a<0$, $1\vee \f{d}{d+s-\sigma}<p<d_\sigma$ and $w\in A_{\f{p}{1\vee \f{d}{d+s-\sigma}}}\cap RH_{(d_\sigma/p)'}$.
\end{enumerate}
\end{thm}

Before proceeding with the proof of the theorem, we need the following technical results on kernel estimates.

Let $D_t(x,y)$ be a kernel of $t\La e^{-t\La} +t\Delta e^{t\Delta}$. We have the following estimates:
\begin{prop}\label{prop-difference}
\begin{enumerate}[(a)]
\item If $a\geq 0$ then
\begin{equation}\label{eq1-DifferenceKernels}
|D_t(x,y)|\lesi t^{-d/2}\left(1+\f{|x|+|y|}{\sqrt{t}}\right)^{-2}e^{-\f{|x-y|^2}{ct}}
\end{equation}
for all $x,y \in \mathbb{R}^d$ and $t>0$.

\item If $-\left(\f{d-2}{2}\right)^2\leq a<0$ then
\begin{equation}\label{eq2-DifferenceKernels}
|D_t(x,y)|\lesi t^{-d/2}\left(1+\f{|x|+|y|}{\sqrt{t}}\right)^{-2}e^{-\f{|x-y|^2}{ct}}
\end{equation}
for all $t>0$ and $|x|, |y|\geq \sqrt{t}/2$.

\end{enumerate}
\end{prop}
We remark that in \cite{K.et.al} the authors gave upper bounds for the kernels of  $e^{-t\La} - e^{t\Delta}$. However, this estimate is not sufficient for us.
\begin{proof}	
We first give the proof for the case $a\geq 0$. Note that in this case since both kernels of $t\La e^{-t\La}$ and $t\Delta e^{t\Delta}$ satisfy Gaussian upper bounds, there exists $C, c>0$ so that
$$
|D_t(x,y)|\leq Ct^{-d/2}e^{-\f{|x-y|^2}{ct}}
$$
for all $x,y \in \mathbb{R}^d$ and $t>0$.

Hence, it suffices to prove \eqref{eq1-DifferenceKernels} for $|x|\sim |y|$ and $|x|, |y|\geq \sqrt{t}/2$. From Duhamel's formula, we obtain that
\begin{equation}
\label{eq-Kato}
\begin{aligned}
D_t(x,y)=&at\int_{\mathbb{R}^d}\tilde{p}_{t/2}(x,z)|z|^{-2}p_{t/2}(z,y)dz+at\int_0^{t/2}\int_{\mathbb{R}^d}\tilde{p}_{t-s,1}(x,z)|z|^{-2}p_{s}(z,y)dz\f{ds}{t-s}\\
&+ at\int_{t/2}^t\int_{\mathbb{R}^d}\tilde{p}_{t-s}(x,z)|z|^{-2}p_{s,1}(z,y)dz\f{ds}{s}\\
&=I_1+I_2+I_3,
\end{aligned}
\end{equation}
where $\tilde{p}_{t,k}(x,y)$ denotes the kernel of $(-1)^k (t\Delta)^ke^{t\Delta}$.

Using the fact that $0\leq p_t(x,y)\leq \tilde{p}_t(x,y)$ to get that
$$
\begin{aligned}
I_1&\lesi t\int_{\mathbb{R}^d}\f{1}{t^{d}}e^{-\f{|x-z|^2}{8t}}|z|^{-2}e^{-\f{|z-y|^2}{8t}}dz\\
&\lesi \f{1}{t^{d/2}}e^{-\f{|x-y|^2}{16t}} t\int_{\mathbb{R}^d}\f{1}{t^{d/2}}e^{-\f{|x-z|^2}{16t}}|z|^{-2}dz
\end{aligned}
$$
which along with the fact that
\begin{equation}\label{eq-V}
\int_{\mathbb{R}^d}\f{1}{t^{d/2}}e^{-\f{|x-z|^2}{ct}}|z|^{-2}dz\lesi \f{1}{|x|^2}
\end{equation}
implies that
$$
I_1\lesi \f{1}{t^{d/2}}e^{-\f{|x-y|^2}{16t}} \f{t}{|x|^2}.
$$

Similarly, by using the Gaussian upper bounds of $\tilde{p}_{t-s,1}(x,z)$ and $p_s(z,y)$ and the fact that
$$
\f{|x-z|^2}{t-s}+\f{|z-y|^2}{s}\geq \f{|x-y|^2}{2t} \ \text{for all $s\in (0,t)$}
$$
we also obtain that
$$
\begin{aligned}
I_2&\lesi e^{-\f{|x-y|^2}{ct}} t\int_0^{t/2}\int_{\mathbb{R}^d}\f{1}{{(t-s)}^{d/2}} \f{1}{s^{d/2}}e^{-\f{|z-y|^2}{c's}}|z|^{-2}dz\f{ds}{t-s}\\
&\lesi \f{1}{{t}^{d/2}}e^{-\f{|x-y|^2}{ct}} t\int_0^{t/2}\int_{\mathbb{R}^d} \f{1}{s^{d/2}}e^{-\f{|z-y|^2}{c's}}|z|^{-2}dz\f{ds}{t}\\
&\lesi \f{1}{t^{d/2}}e^{-\f{|x-y|^2}{ct}} \f{t}{|y|^2}
\end{aligned}
$$
where in the last inequality we used \eqref{eq-V}.

Similarly, by a change of variable and arguing as in $I_2$,
$$
I_3\lesi \f{1}{t^{d/2}}e^{-\f{|x-y|^2}{16t}} \f{t}{|x|^2}.
$$
This completes the proof for the case $a\geq 0$.

\medskip

We now consider the case when $-\left(\f{d-2}{2}\right)^2\leq a<0$. In this situation $\sigma>0$ and thus it is easy to observe that
$$
|D_t(x,y)|\lesi t^{-d/2}e^{-\f{|x-y|^2}{ct}},
$$
whenever $|x|,|y|\geq \sqrt{t}/2$. Hence, it suffices to prove \eqref{eq2-DifferenceKernels} for $|x|,|y|\geq \sqrt{t}$ and $|x|\sim |y|$.

By expressing $D_t(x,y)$ as in \eqref{eq-Kato}, we will need to estimate $I_1, I_2, I_3$ for $-\left(\f{d-2}{2}\right)^2<a<0$ and $|x|,|y|\geq \sqrt{t}$ and $|x|\sim |y|$.

Arguing similarly to the case $a\geq 0$, we have
$$
I_1\lesi \f{1}{t^{d/2}}e^{-\f{|x-y|^2}{16t}} \f{t}{|x|^2}.
$$

For the second term $I_2$, from the kernel bounds estimates of $\tilde{p}_{t,1}(x,y)$ and $p_t(x,y)$, and arguing similarly to the case 1, we obtain
$$
\begin{aligned}
I_2&\lesi e^{-\f{|x-y|^2}{ct}} t\int_0^{t/2}\int_{\mathbb{R}^d}\f{1}{{(t-s)}^{d/2}} \f{1}{s^{d/2}}e^{-\f{|z-y|^2}{c's}}|z|^{-2}\left(1+\f{\sqrt{s}}{|z|}\right)^\sigma dz\f{ds}{t}\\
&\lesi \f{1}{{t}^{d/2}}e^{-\f{|x-y|^2}{ct}} \int_0^{t/2}\int_{\mathbb{R}^d} \f{1}{s^{d/2}}e^{-\f{|z-y|^2}{c's}}|z|^{-2}\left(1+\f{\sqrt{s}}{|z|}\right)^\sigma dzds\\
&\lesi \f{1}{{t}^{d/2}}e^{-\f{|x-y|^2}{ct}} \int_0^{t/2}\int_{|z|\geq \sqrt{s}} \f{1}{s^{d/2}}e^{-\f{|z-y|^2}{c's}}|z|^{-2}\left(1+\f{\sqrt{s}}{|z|}\right)^\sigma dzds\\
&\quad +\f{1}{{t}^{d/2}}e^{-\f{|x-y|^2}{ct}} \int_0^{t/2}\int_{|z|<\sqrt{s}} \f{1}{s^{d/2}}e^{-\f{|z-y|^2}{c's}}|z|^{-2}\left(1+\f{\sqrt{s}}{|z|}\right)^\sigma dzds.
\end{aligned}
$$
Similarly to the case $a\geq 0$, we have
$$
\begin{aligned}
\f{1}{{t}^{d/2}}e^{-\f{|x-y|^2}{ct}} &\int_0^{t/2}\int_{|z|\geq \sqrt{s}} \f{1}{s^{d/2}}e^{-\f{|z-y|^2}{c's}}|z|^{-2}\left(1+\f{\sqrt{s}}{|z|}\right)^\sigma dzds\\
&\lesi \f{1}{{t}^{d/2}}e^{-\f{|x-y|^2}{ct}} \int_0^{t/2}\int_{|z|\geq \sqrt{s}} \f{1}{s^{d/2}}e^{-\f{|z-y|^2}{c's}}|z|^{-2}dzds\\
&\lesi \f{1}{t^{d/2}}e^{-\f{|x-y|^2}{16t}} \f{t}{|y|^2}\sim \f{1}{t^{d/2}}e^{-\f{|x-y|^2}{16t}} \f{t}{|x|^2}.
\end{aligned}
$$
where in the second inequality we used \eqref{eq-V}.

From the fact that $\sigma+2<d$ we have
$$
\int_{\mathbb{R}^d}\f{1}{t^{d/2}}e^{-\f{|x-z|^2}{ct}}|z|^{-2}\left(1+\f{\sqrt{s}}{|z|}\right)^\sigma dz\lesi \f{1}{|x|^2}.
$$
This implies that
$$
\begin{aligned}
\f{1}{{t}^{d/2}}e^{-\f{|x-y|^2}{ct}} \int_0^{t/2}\int_{|z|<\sqrt{s}} \f{1}{s^{d/2}}e^{-\f{|z-y|^2}{c's}}|z|^{-2}\left(1+\f{\sqrt{s}}{|z|}\right)^\sigma dzds
&\lesi \f{1}{t^{d/2}}e^{-\f{|x-y|^2}{16t}} \f{t}{|x|^2}.
\end{aligned}
$$

Likewise, we get that
$$
I_3\lesi \f{1}{t^{d/2}}e^{-\f{|x-y|^2}{16t}} \f{t}{|x|^2}.
$$
This completes our proof.
\end{proof}

\begin{proof}[Proof of Theorem \ref{thm-difference}:]

We consider two cases.

\noindent{\bf Case 1: $a\geq 0$}

Fix $1<p<\vc$ and $w\in A_p$. Observe that by Proposition \ref{prop-difference}
$$
\begin{aligned}
\Big(\int_0^\vc &t^{-s}\left|\left(t\La e^{-t\La} +t\Delta e^{t\Delta}\right)f(x)\right|^2\f{dt}{t}\Big)^{1/2}\\
&\leq \left(\sum_{j\in \mathbb{Z}}\int_{2^{2j}}^{2^{2(j+1)}} t^{-s}\left|\left(t\La e^{-t\La} +t\Delta e^{t\Delta}\right)f(x)\right|^2\f{dt}{t}\right)^{1/2}\\
&\leq \left[\sum_{j\in \mathbb{Z}}\int_{2^{2j}}^{2^{2(j+1)}} t^{-s}\left(\int_{\mathbb{R}^d} |D_t(x,y)|\,|f(y)|dy\right)^2\f{dt}{t}\right]^{1/2}\\
&\leq \left[\sum_{j\in \mathbb{Z}}\int_{2^{2j}}^{2^{2(j+1)}} 2^{-2js}\left(\int_{\mathbb{R}^d}2^{-jd}\left(1+\f{|x|+|y|}{2^j}\right)^{-2}e^{-\f{|x-y|^2}{c2^{2j}}}|f(y)|dy\right)^2\f{dt}{t}\right]^{1/2}\\
&\leq \sum_{j\in \mathbb{Z}}2^{-js}\int_{\mathbb{R}^d}2^{-jd}\left(1+\f{|x|+|y|}{2^j}\right)^{-2}e^{-\f{|x-y|^2}{c2^{2j}}}|f(y)|dy\\
\end{aligned}
$$
where in the last inequality we used the fact that $\ell_1\hookrightarrow\ell_2$.

As in \cite{K.et.al} we split the right hand side term above into two terms with respect to low-energy and high-energy cases.
$$
\begin{aligned}
\Big(\int_0^\vc &t^{-s}\left|\left(t\La e^{-t\La} +t\Delta e^{t\Delta}\right)f(x)\right|^2\f{dt}{t}\Big)^{1/2}\\
&\leq\int_{\mathbb{R}^d}\sum_{j\in \mathbb{Z}: 2^j> |x|+|y|}2^{-j(d+s)}\left(1+\f{|x|+|y|}{2^j}\right)^{-2}e^{-\f{|x-y|^2}{c2^{2j}}}|f(y)|dy\\
& \quad +\int_{\mathbb{R}^d}\sum_{j\in \mathbb{Z}: 2^j\leq |x|+|y|}2^{-j(d+s)}\left(1+\f{|x|+|y|}{2^j}\right)^{-2}e^{-\f{|x-y|^2}{c2^{2j}}}|f(y)|dy\\
:=I_1(x)+I_2(x).
\end{aligned}
$$
For the first term, we have
$$
\begin{aligned}
I_1(x)&\leq \int_{\mathbb{R}^d}\sum_{j\in \mathbb{Z}: 2^j> |x|+|y|}2^{-j(d+s)}|f(y)|dy \lesi \int_{\mathbb{R}^d}\f{1}{(|x|+|y|)^{d+s}}|f(y)|dy\\
&\lesi \int_{|y|\leq |x|}\f{1}{(|x|+|y|)^{d+s}}|f(y)|dy+\int_{|y|> |x|}\f{1}{(|x|+|y|)^{d+s}}|f(y)|dy\\
&\lesi I_{11}(x)+I_{12}(x),
\end{aligned}
$$
which implies
$$
\|I_1(\cdot)\|_{L^p_w}\lesi \|I_{11}(\cdot)\|_{L^p_w}+\|I_{12}(\cdot)\|_{L^p_w}.
$$

It is easy to see that
$$
I_{11}(x)\leq \int_{|y|\leq |x|}\f{|y|^s}{|x|^{d+s}}\f{|f(y)|}{|y|^s}dy\leq \mathcal{M}\left(\f{f}{|\cdot|^s}\right)(x)
$$
which yields that
$$
\|I_{11}(\cdot)\|_{L^p_w}\lesi \left\|\f{f}{|x|^s}\right\|_{L^p_w}.
$$
Taking $g\in L^{p'}(v), v=w^{1-p'}\in A_{p'}$ we then have
$$
\begin{aligned}
\langle I_{12}(\cdot), g\rangle &\leq \int_{\mathbb{R}^d}\int_{|y|> |x|}\f{1}{(|x|+|y|)^{d+s}}|f(y)|\, |g(x)|dydx\\
&\leq \int_{\mathbb{R}^d}\int_{|y|> |x|}\f{1}{|y|^{d}}\f{|f(y)|}{|y|^s}\, |g(x)|dydx\\
&\leq \int_{\mathbb{R}^d}\int_{|x|<|y|}\f{1}{|y|^{d}}\f{|f(y)|}{|y|^s}\, |g(x)|dx dy\lesi \int_{\mathbb{R}^d}\left|\mathcal{M}g(y)\right|\left|\f{f(y)}{|y|^s}\right|dy\\
&\lesi  \left\|\f{f}{|\cdot|^s}\right\|_{L^p_w}\|\mathcal{M}g\|_{L^{p'}_v}\lesi \left\|\f{f}{|x|^s}\right\|_{L^p_w}\|g\|_{L^{p'}_v}.
\end{aligned}
$$
As a consequence,
$$
\|I_{12}(\cdot)\|_{L^p_w}\lesi \left\|\f{f}{|x|^s}\right\|_{L^p_w}.
$$

We turn to the second term $I_2(x)$. For $0<\epsilon<\f{2-s}{2}$, we have
$$
\begin{aligned}
I_2(x)&\lesi \int_{\mathbb{R}^d}\sum_{j\in \mathbb{Z}: 2^j\leq |x|+|y|}\f{1}{|x-y|^{d-\epsilon}}\f{1}{(|x|+|y|)^{s+\epsilon}}\left(\f{|x|+|y|}{2^j}\right)^{-2+s+\epsilon} |f(y)|dy\\
&\lesi \int_{\mathbb{R}^d}\f{1}{|x-y|^{d-\epsilon}}\f{1}{(|x|+|y|)^{s+\epsilon}}|f(y)|dy\lesi \int_{\mathbb{R}^d}\f{1}{|x-y|^{d-\epsilon}}\f{1}{(|x|+|y|)^{\epsilon}}\f{|f(y)|}{|y|^s}dy\\
&:= \int_{\Gamma_1(x)}\ldots+\int_{\Gamma_2(x)}\ldots+\int_{\Gamma_3(x)}\ldots+\int_{\Gamma_4(x)}\ldots\\
&:= I_{21}(x)+I_{22}(x)+I_{23}(x)+I_{24}(x),
\end{aligned}
$$
where $\Gamma_1(x)=\{y: |y|<|x/2|\}$, $\Gamma_2(x)=\{y: |y|\geq 2|x|\}$, $\Gamma_3(x)=\{y: |x|/2\leq   |y|<2|x|\}\cap B(x, |x|/2)$ and $\Gamma_4(x)=\{y: |x|/2\leq|y|<2|x|\}\cap B(x, |x|/2)^c$.

It is easy to dominate $I_{21}(x)$ as follows
$$
\begin{aligned}
I_{21}(x)\lesi \int_{|y|<|x|/2}\f{1}{|x|^{d}} \f{|f(y)|}{|y|^s}dy\lesi \mathcal{M}\left(\f{f}{|\cdot|^s}\right)(x),
\end{aligned}
$$
which implies that
$$
\|I_{21}(\cdot)\|_{L^p_w}\lesi \left\|\f{f}{|x|^s}\right\|_{L^p_w}.
$$
For the term $I_{22}(x)$, we have
$$
\begin{aligned}
I_{22}(x)\lesi \int_{|y|\geq 2|x|}\f{1}{|y|^{d}} \f{|f(y)|}{|y|^s}dy.
\end{aligned}
$$
At this stage, by using the argument in the estimate $I_{12}(\cdot)$ we also get that
$$
\|I_{22}(\cdot)\|_{L^p_w}\lesi \left\|\f{f}{|x|^s}\right\|_{L^p_w}.
$$

In addition, we have
$$
I_{23}(x)\lesi \int_{B(x, 2|x|)}\f{1}{|x|^d}\f{|f(y)|}{|y|^s}dy\lesi \mathcal{M}\left(\f{f}{|\cdot|^s}\right)(x),
$$
which implies that
$$
\|I_{23}(\cdot)\|_{L^p_w}\lesi \left\|\f{f}{|x|^s}\right\|_{L^p_w}.
$$

The last term $I_{24}$ can be dealt with as follows.
$$
\begin{aligned}
I_{24}(x)&\lesi \int_{B(x, |x|/2)^c}\f{1}{|x-y|^{d-\epsilon} |x|^\epsilon}\f{|f(y)|}{|y|^s}dy\\
&\lesi \sum_{j=0}^\vc \int_{2^{-j-1}|x|\leq |x-y|<2^{-j}|x|}\f{1}{|x-y|^{d-\epsilon} |x|^\epsilon}\f{|f(y)|}{|y|^s}dy\\
&\lesi \sum_{j=0}^\vc 2^{-j\epsilon}\mathcal{M}\left(\f{f}{|\cdot|^s}\right)(x)\lesi \mathcal{M}\left(\f{f}{|\cdot|^s}\right)(x)
\end{aligned}
$$
which implies that
$$
\|I_{24}(\cdot)\|_{L^p_w}\lesi \left\|\f{f}{|x|^s}\right\|_{L^p_w}.
$$

\bigskip

\noindent{\bf Case 2: $-\left(\f{d-2}{2}\right)^2\leq a< 0$}

Fix $p\in \left(1\vee \f{d}{d+s-\sigma},d_\sigma\right)$ and $w\in A_{\f{p}{1\vee \f{d}{d+s-\sigma}}}\cap RH_{\left(\f{d_\sigma}{p}\right)'}$. Hence, there exist $p_1, q_1$ so that $1\vee \f{d}{d+s-\sigma}<p_1<p<q_1<d_\sigma$, and $w\in A_{\f{p}{p_1}}\cap RH_{\left(\f{q_1}{p}\right)'}$.
Similarly to Case 1, by Proposition \ref{prop-difference} we obtain that
$$
\begin{aligned}
\Big(\int_0^\vc &t^{-s}\left|\left(t\La e^{-t\La} +t\Delta e^{t\Delta}\right)f(x)\right|^2\f{dt}{t}\Big)^{1/2}\\
&\leq\int_{\mathbb{R}^d}\sum_{j\in \mathbb{Z}: 2^j\geq  (|x|\wedge |y|)/2}2^{-j(d+s)}\left(1+\f{2^j}{|x|}\right)^{\sigma}\left(1+\f{2^j}{|y|}\right)^{\sigma}e^{-\f{|x-y|^2}{c2^{2j}}}|f(y)|dy\\
& \quad +\int_{\mathbb{R}^d}\sum_{j\in \mathbb{Z}: 2^j< (|x|\wedge |y|)/2}2^{-j(d+s)}\left(1+\f{|x|+|y|}{2^j}\right)^{-2}e^{-\f{|x-y|^2}{c2^{2j}}}|f(y)|dy\\
&:=J_1(x)+J_2(x).
\end{aligned}
$$
The argument used to estimate $I_2(x)$ in Case 1 also shows that
$$
\|J_2(\cdot)\|_{L^p_w}\lesi \left\|\f{f}{|x|^s}\right\|_{L^p_w}.
$$
It remains to show that
$$
\|J_1(\cdot)\|_{L^p_w}\lesi \left\|\f{f}{|x|^s}\right\|_{L^p_w}.
$$
Indeed, we have
$$
\begin{aligned}
J_1(x)&\leq\int_{\mathbb{R}^d}\sum_{j\in \mathbb{Z}: 2^j\geq  (|x|\vee |y|)/2}2^{-j(d+s)}\left(1+\f{2^j}{|x|}\right)^{\sigma}\left(1+\f{2^j}{|y|}\right)^{\sigma}e^{-\f{|x-y|^2}{c2^{2j}}}|f(y)|dy\\
& \quad +\int_{\mathbb{R}^d}\sum_{j\in \mathbb{Z}: |y|/2>2^j\geq  |x|/2}2^{-j(d+s)}\left(1+\f{2^j}{|x|}\right)^{\sigma}\left(1+\f{2^j}{|y|}\right)^{\sigma}e^{-\f{|x-y|^2}{c2^{2j}}}|f(y)|dy\\
& \quad +\int_{\mathbb{R}^d}\sum_{j\in \mathbb{Z}: |x|/2>2^j\geq  |y|/2}2^{-j(d+s)}\left(1+\f{2^j}{|x|}\right)^{\sigma}\left(1+\f{2^j}{|y|}\right)^{\sigma}e^{-\f{|x-y|^2}{c2^{2j}}}|f(y)|dy\\
&:= J_{11}(x)+J_{12}(x)+J_{13}(x).
\end{aligned}
$$
For the term $J_{11}$, one has
$$
\begin{aligned}
J_{11}(x)&\leq\int_{\mathbb{R}^d}\f{1}{(|x|+|y|)^{d+s-2\sigma}|x|^\sigma |y|^\sigma}|f(y)|dy\\
&\lesi \int_{|y|\leq |x|}\ldots +\int_{|y|\geq |x|}\ldots\\
&:=J^1_{11}(x)+J^2_{11}(x).
\end{aligned}
$$
We now consider the contribution of $J^1_{11}(x)$. In this case, we have
$$
\begin{aligned}
J^1_{11}(x)&\leq\int_{|y|\leq |x|}\f{|y|^{s-\sigma}}{|x|^{d+s-\sigma} }\f{|f(y)|}{|y|^s}dy\\
\end{aligned}
$$
If $s-\sigma\geq 0$ then
$$
J^1_{11}(x)\leq\int_{|y|\leq |x|}\f{1}{|x|^{d} }\f{|f(y)|}{|y|^s}dy\lesi \mathcal{M}\left(\f{f}{|\cdot|}\right)(x),
$$
and hence,
$$
\|J^1_{11}(\cdot)\|_{L^p_w}\lesi \left\|\f{f}{|x|^s}\right\|_{L^p_w}.
$$
Otherwise, if $s-\sigma< 0$ then by H\"older's inequality
$$
\begin{aligned}
J^1_{11}(x) &\lesi \left(\int_{|y|\leq |x|} \f{|y|^{(s-\sigma)p_1'}}{|x|^{(d+s-\sigma)p_1'} }dy\right)^{1/p_1'}\left(\int_{|y|\leq |x|} \f{|f(y)|^{p_1}}{|y|^{sp_1}}dy\right)^{1/p_1}.
\end{aligned}
$$
Since $0<(\sigma-s)p_1'<d$, by Lemma \ref{lem1-Tt}, we have
$$
\begin{aligned}
J^1_{11}(x) &\lesi \left(\f{1}{|x|^d}\int_{|y|\leq |x|} \f{|f(y)|^{p_1}}{|y|^{sp_1}}dy\right)^{1/p_1}\lesi \mathcal{M}_{p_1}\left(\f{f}{|\cdot|^s}\right)(x)
\end{aligned}
$$
which implies that
$$
\|J^1_{11}(\cdot)\|_{L^p_w}\lesi \left\|\f{f}{|x|^s}\right\|_{L^p_w}.
$$
To estimate the term $J^2_{11}$ we employ a duality argument. Set $v=w^{1-p'}$, and hence by (vii) Lemma \ref{weightedlemma1},  $v\in A_{p'/q_1'}$. For $g\in L^{p'}_v$ we have
$$
\begin{aligned}
\langle J^2_{11}, g\rangle&\lesi \int_{\mathbb{R}^d}\int_{|y|\geq |x|}\f{1}{(|y|)^{d-\sigma}|x|^\sigma} \f{|f(y)|}{|y|^s}dy|g(x)|dx\\
& \lesi\int_{\mathbb{R}^d}\int_{|x|\leq |y|} \f{|g(x)|}{|x|^\sigma}dx\f{|f(y)|}{|y|^{d+s-\sigma}}dy\\
& \lesi\int_{\mathbb{R}^d}\left(\int_{|x|\leq |y|}|g(x)|^{q_1'} dx\right)^{1/q'_1}\left(\int_{|x|\leq |y|}\f{1}{|x|^{\sigma q_1}}dx\right)^{1/q_1} \f{|f(y)|}{|y|^{d+s-\sigma}}dy
\end{aligned}
$$
which, together with Lemma \ref{lem1-Tt} and Lemma \ref{Lem-maximalfunction}, gives
$$
\begin{aligned}
\langle J^2_{11}, g\rangle
&\lesi \int_{\mathbb{R}^d}\left(\f{1}{|y|^d}\int_{|x|\leq |y|}|g(x)|^{q_1'} dx\right)^{1/q'_1}\f{|f(y)|}{|y|^s}dy\\
&\lesi \left\langle \mathcal{M}_{q_1'}g, \f{f}{|\cdot|^s}\right\rangle\\
&\lesi \|\mathcal{M}_{q_1'}g\|_{L^{p'}_v}\left\|\f{f}{|\cdot|^s}\right\|_{L^p_w}\lesi \|g\|_{L^{p'}_v}\left\|\f{f}{|\cdot|^s}\right\|_{L^p_w}.
\end{aligned}
$$
Hence,
$$
\|J^2_{11}(\cdot)\|_{L^p_w}\lesi \left\|\f{f}{|x|^s}\right\|_{L^p_w}.
$$
Let us move on the term $J_{12}$. We split this term as follows.
$$
\begin{aligned}
J_{12}(x)&\lesi \int_{|y|\geq 2|x|}\sum_{\substack{ j\in \mathbb{Z}: |y|/2>2^j\geq  |x|/2}}2^{-j(d+s)}\left(1+\f{2^j}{|x|}\right)^{\sigma}\left(1+\f{2^j}{|y|}\right)^{\sigma}e^{-\f{|x-y|^2}{c2^{2j}}}|f(y)|dy\\
&+\int_{ |y|<2|x|}\sum_{\substack{j\in \mathbb{Z}: |y|/2>2^j\geq  |x|/2}}2^{-j(d+s)}\left(1+\f{2^j}{|x|}\right)^{\sigma}\left(1+\f{2^j}{|y|}\right)^{\sigma}e^{-\f{|x-y|^2}{c2^{2j}}}|f(y)|dy\\
&=J^1_{12}(x)+J^2_{12}(x).
\end{aligned}
$$
It is easy to see that
$$
\begin{aligned}
J^2_{12}(x)&\lesi \int_{|y|< 2|x|}\sum_{\substack{j\in \mathbb{Z}: |y|/2>2^j\geq  |x|/2}}2^{-j(d+s)}\left(1+\f{2^j}{|x|}\right)^{\sigma}\left(1+\f{2^j}{|y|}\right)^{\sigma}e^{-\f{|x-y|^2}{c2^{2j}}}|f(y)|dy\\
&\lesi \int_{|y|<2|x|}\f{1}{|x|^{d+s}}|f(y)|dy\lesi \mathcal{M}\left(\f{f}{|\cdot|^s}\right)(x)
\end{aligned}
$$
which yields that
$$
\|J^2_{12}(\cdot)\|_{L^p_w}\lesi \left\|\f{f}{|x|^s}\right\|_{L^p_w}.
$$
To consider the contribution of $J^1_{12}$, we write
$$
\begin{aligned}
J^1_{12}(x)&\lesi \int_{|y|\geq 2|x|}\sum_{\substack{ j\in \mathbb{Z}: |y|/2>2^j\geq  |x|/2}}2^{-j(d+s)}\left(1+\f{2^j}{|x|}\right)^{\sigma}\left(1+\f{2^j}{|y|}\right)^{\sigma}e^{-\f{|x-y|^2}{c2^{2j}}}|f(y)|dy.
\end{aligned}
$$
In this case, we have $\left(1+\f{2^j}{|y|}\right)^{\sigma}\lesi 1$. Hence
$$
\begin{aligned}
J^1_{12}(x)&\lesi \int_{|y|\geq 2|x|}\f{1}{|x-y|^{d+s-\sigma}|x|^\sigma}|f(y)|dy\sim \int_{|y|\geq 2|x|}\f{1}{|y|^{d+s-\sigma}|x|^\sigma}|f(y)|dy.
\end{aligned}
$$
At this stage, arguing similarly to the estimate of $J_{11}^2$ we also obtain that

$$
\|J_{12}(\cdot)\|_{L^p_w}\lesi \left\|\f{f}{|x|^s}\right\|_{L^p_w}.
$$

It remains to show that
$$
\|J_{13}(\cdot)\|_{L^p_w}\lesi \left\|\f{f}{|x|^s}\right\|_{L^p_w}.
$$
The proof of this estimate can be done in the same manner as that of $J_{12}$. We leave it to the interested reader.

This completes our proof.
\end{proof}

\subsection{Proof of Theorems \ref{mainthm} and Theorem \ref{the:mainthm3}}
\begin{proof}[Proof of Theorem \ref{mainthm}:]
	Fix $0<s<2$, $d'_\sigma<p<d_{s+\sigma}$ and $w\in A_{p/d'_{\sigma}}\cap RH_{(d_{s+\sigma}/p)'}$. Then by Theorem \ref{thm-square}, Theorem \ref{thm-difference} and Theorem \ref{thm-HardyIneq} we have
	$$
	\begin{aligned}
	\|(-\Delta)^{s/2}f\|_{L^p_w}&\lesi \left\|\left(\int_0^\vc t^{-s}|t(-\Delta)e^{t\Delta}f|^2\f{dt}{t}\right)^{1/2}\right\|_{L^p_w}\\
	&\lesi \left\|\left(\int_0^\vc t^{-s}\left|(t\La e^{-t\La}+t\Delta e^{t\Delta})f\right|^2\f{dt}{t}\right)^{1/2}\right\|_{L^p_w}+\left\|\left(\int_0^\vc t^{-s}|t\La e^{-t\La}f|^2\f{dt}{t}\right)^{1/2}\right\|_{L^p_w}\\
	&\lesi \left\|\f{f}{|x|^s}\right\|_{L^p_w}+\|\La^{s/2}f\|_{L^p_w}\\
	&\lesi \|\La^{s/2}f\|_{L^p_w}.
	 \end{aligned}
	$$
	Conversely, for $1<p<\vc$ with $p_1:=1\vee \f{d}{d-\sigma}<p<\f{d}{s\vee \sigma}:=p_2$ and $w\in A_{p/p_1}\cap RH_{(p_2/p)'}$ we have
	$$
	\begin{aligned}
	\|\La^{s/2}f\|_{L^p_w}&\lesi \left\|\left(\int_0^\vc t^{-s}|t\La e^{-t\La}f|^2\f{dt}{t}\right)^{1/2}\right\|_{L^p_w} \\
	&\lesi \left\|\left(\int_0^\vc t^{-s}\left|(t\La e^{-t\La}+t\Delta e^{t\Delta})f\right|^2\f{dt}{t}\right)^{1/2}\right\|_{L^p_w}+\left\|\left(\int_0^\vc t^{-s}|t(-\Delta)e^{t\Delta}f|^2\f{dt}{t}\right)^{1/2}\right\|_{L^p_w}\\
	&\lesi \left\|\f{f}{|x|^s}\right\|_{L^p_w}+\|(-\Delta)^{s/2}f\|_{L^p_w}\\
	&\lesi \|(-\Delta)^{s/2}f\|_{L^p_w},
	\end{aligned}
	$$
	where in the last inequality we used Theorem \ref{thm-HardyIneq}.
	
	This completes our proof.
	
	\end{proof}


\begin{proof}[Proof of Theorem \ref{the:mainthm3}]
	Before starting the proof, we note that the flow 
	$e^{it \mathcal{L}_{a}}$ satisfies for all $s$
	the conservation laws
	\begin{equation*}
	  \|\mathcal{L}_{a}^{s/2}e^{it \mathcal{L}_{a}}u_{0}\|_{L^{2}}=
	  \|\mathcal{L}_{a}^{s/2}u_{0}\|_{L^{2}}
	  \qquad
	  \forall t\in \mathbb{R}
	\end{equation*}
	by self-adjointness. 
	By an elementary application of 
	Theorem \ref{mainthm} in the unweighted case,
  this implies the almost conservation of $H^{s}$ norms
	\begin{equation}\label{eq:energy}
	  \|(-\Delta)^{s/2}e^{it \mathcal{L}_{a}}u_{0}\|_{L^{2}}
	  \simeq
	  \|(-\Delta)^{s/2}u_{0}\|_{L^{2}}
	  \quad\text{for}\quad 0<s<2.
	\end{equation}

	Consider a generic Schr\"{o}dinger equation with potential
	\begin{equation*}
	  iu_{t}+\Delta u-c(x)u=0, \ \ c(x)=\f{a}{|x|^2}.
	\end{equation*}
	If $u(t,x)$ solves this equation,
  then the following identity holds for any sufficiently smooth
  $\psi:\mathbb{R}^{d}\to \mathbb{R}$:
  \begin{equation}\label{eq:virial}
    \Re \nabla \cdot Q
    +
    \Im\partial_{t}\{\bar{u}\nabla \psi \cdot\nabla u\}
    =
    \textstyle 
    -\frac12 \Delta^{2}\psi|u|^{2}
    +2 \sum_{j,k=1}^{d}
    \partial_{j}u \partial_{j}\partial_{k}\psi \partial_{k}\bar{u}
    -\nabla \psi \cdot\nabla c |u|^{2}
  \end{equation}
  where
  \begin{equation*}
    \textstyle
    Q=
    \nabla u (\overline{u}\Delta \psi+
    \nabla \psi \cdot \nabla \overline{u})
    -\frac12 \nabla\Delta \psi|u|^{2}
    -\nabla\psi 
    \left[c|u|^{2}-iu_{t}\overline{u}
      +|\nabla u |^{2}\right].
  \end{equation*}
 Formula \eqref{eq:virial} is usually called a \emph{virial}
 (or \emph{Morawetz}) identity, and it is easy to check directly
 by expanding the derivative $\nabla \cdot Q$, and using
 the equation for $u(t,x)$
 (see e.g.~\cite{BRV}, \cite{DF}, 
 \cite{DR}).

 Consider the case when the weight $\psi$ is a radial function;
 by abuse of notation we use the same symbol
 $\psi(x)=\psi(|x|)$. Then we can write
 \begin{equation*}
   \nabla \psi \cdot \nabla c=
   \psi' \partial_{r}c
 \end{equation*}
 where $\partial_{r}=\frac{x}{|x|}\cdot \nabla$ is the radial
 derivative. If we denote by 
 $\nabla^{T}u=\nabla u-\frac{x}{|x|}\partial_{r}u$ the tangential
 component of $\nabla u$, we have also the identities
 \begin{equation*}
   |\nabla u|^{2}=|\partial_{r} u|^{2}+|\nabla^{T}u|^{2}
   \quad\text{and}\quad 
   \sum_{j,k=1}^{d}
       \partial_{j}u \partial_{j}\partial_{k}\psi \partial_{k}\bar{u}=
   \psi''|\partial_{r}u|^{2}+\frac{\psi'}{|x|}|\nabla^{T}u|^{2}.
 \end{equation*}
 Then formula \eqref{eq:virial} reduces to
 \begin{equation}\label{eq:virial2}
   \Re \nabla \cdot Q
   +
   \Im\partial_{t}\{\bar{u}\psi' \partial_{r} u\}
   =
   \textstyle 
   2\psi''|\partial_{r}u|^{2}+2\frac{\psi'}{|x|}|\nabla^{T}u|^{2}
   -\frac12 \Delta^{2}\psi|u|^{2}
   -\psi'\partial_{r} c |u|^{2}.
 \end{equation}

 We now pick an explicit radial weight
 \begin{equation*}
   \psi(r)=\int_{0}^{r}\frac{s^{\epsilon}}{1+s^{\epsilon}}ds,
   \qquad
   0<\epsilon<1
 \end{equation*}
 where the parameter $\epsilon$ will be chosen later.
 A straightforward computation gives, writing for brevity
 $r=|x|$,
 \begin{equation*}
   2\psi''|\partial_{r}u|^{2}+2\frac{\psi'}{|x|}|\nabla^{T}u|^{2}=
   \frac{2 \epsilon r^{\epsilon}}{(1+r^{\epsilon})^{2}}
   \frac1r |\partial_{r}u|^{2}+
   \frac{2r^{\epsilon}}{1+r^{\epsilon}}\frac1r|\nabla^{T}u|^{2}
   \ge
   \frac{2 \epsilon r^{\epsilon}}{(1+r^{\epsilon})^{2}}
   \frac1r |\nabla u|^{2},
 \end{equation*}
 \begin{equation*}
   -\psi'\partial_{r} c |u|^{2}=
   \frac{r^{\epsilon}}{1+r^{\epsilon}}
   \frac{2a}{r^{3}}|u|^{2}
 \end{equation*}
 and
 \begin{equation*}
   -\frac12 \Delta^{2}\psi|u|^{2}=
   \frac{1}{r^{3}}
   \frac{r^{\epsilon}}{1+r^{\epsilon}}
   \left[
   	\frac{\mu_{d}}{2}+
   	\epsilon \cdot \beta(r^{\epsilon})
   \right]|u|^{2},
   \qquad
   \mu_{d}=(d-1)(d-3),
 \end{equation*}
 where
 \begin{equation*}
   2\beta(r)=
   -\frac{d^{2}-6d+7}{1+r}+
   \epsilon (2d-5)\frac{r^{2}-1}{(1+r)^{3}}-
   \epsilon^{2}\frac{r^{2}-4r+1}{(1+r)^{3}}
   \quad\implies\quad |\beta(r)|\le 3d^{2}
 \end{equation*}
 since $d\ge3$.
 We substitute these expressions into \eqref{eq:virial2} and
 integrate over $\mathbb{R}^{d}$; if $u$ is a $H^{1}$ solution
 the term in divergence form $\nabla Q$ vanishes, and
 we obtain
 \begin{equation*}
  \Im\int\partial_{t}\{\bar{u}\psi' \partial_{r} u\}dx
  \ge
  \int
  \frac{2 \epsilon r^{\epsilon}}{(1+r^{\epsilon})^{2}}
  \frac{|\nabla u|^{2}}{r}dx+
  \int
  \left[
    2a+
  	\frac{\mu_{d}}{2}-
  	3d^{2}\epsilon
  \right]
  \frac{r^{\epsilon}}{1+r^{\epsilon}}
  \frac{|u|^{2}}{r^{3}}dx.
 \end{equation*}
 Recall that by assumption we have
 \begin{equation*}
 	 \delta:=
   a+\left(\frac{d-2}{2}\right)^{2}-\frac14
   =
   a+ \frac{\mu_{d}}{4}>0,
 \end{equation*}
 thus if we choose $\epsilon=\min\{1,d^{-2}\delta/3\}$ 
 and note that $\epsilon<\delta$ we have proved the inequality
 \begin{equation}\label{eq:intid}
 	\Im\int\partial_{t}\{\bar{u}\psi' \partial_{r} u\}dx
 	\ge
 	\int
 	\frac{\epsilon r^{\epsilon}}{(1+r^{\epsilon})^{2}}
 	\frac{|\nabla u|^{2}}{r}dx+
 	\int
 	\frac{\epsilon r^{\epsilon}}{1+r^{\epsilon}}
 	\frac{|u|^{2}}{r^{3}}dx.
 \end{equation}
 We next integrate \eqref{eq:intid} with respect to time on the
 interval $t\in[0,T]$; we obtain
 \begin{equation}\label{eq:intxt}
 	 \left.
   \Im\int\bar{u}\psi' \partial_{r} u dx\right|_{t=0}^{t=T}
   \ge
   \int_{0}^{T}
   \int
   \left[
   	\frac{\epsilon r^{\epsilon}}{(1+r^{\epsilon})^{2}}
   	\frac{|\nabla u|^{2}}{r}d
   	+
   	\frac{\epsilon r^{\epsilon}}{1+r^{\epsilon}}
   	\frac{|u|^{2}}{r^{3}}
   \right]
   dxdt.
 \end{equation}
 We note that the bilinear form
 \begin{equation*}
   B(v,w)=\int \overline{v(x)}\psi'(|x|)\partial_{r} w(x)dx
 \end{equation*}
 satisfies the estimate
 \begin{equation}\label{eq:H12}
   |B(u,v)|\le 3\|v\|_{\dot H^{1/2}}\|w\|_{\dot H^{1/2}},
 \end{equation}
 where we used the notation 
 \begin{equation*}
   \|v\|_{\dot H^{s}}:=\|(-\Delta)^{s/2}v\|_{L^{2}}.
 \end{equation*}
 Indeed, by Cauchy-Schwartz and the inequality $|\psi'|\le1$ we have
 \begin{equation*}
   |B(v,w)|\le\|v\|_{L^{2}}\|w\|_{\dot H^{1}}
 \end{equation*}
 On the other hand, integrating by parts we have
 \begin{equation*}
   B(v,w)=\int \overline{v(x)}\nabla\psi(x)\nabla w(x)dx=
   \int
   [
   w\overline{\nabla v}\nabla \psi +
   w\overline{v}\Delta \psi
   ]dx.
 \end{equation*}
 Note that $|\nabla \psi|\le1$ while
 \begin{equation*}
   |\Delta \psi|=\frac{\epsilon r^{\epsilon-1}}{(1+r^{\epsilon})^{2}}+
   \frac{(d-1) r^{\epsilon-1}}{1+r^{\epsilon}}
   \le
   \frac{d}{r}
 \end{equation*}
 so that, by Hardy's inequality,
 \begin{equation*}
   \left|\int w\overline{v}\Delta \psi dx\right|\le
   d\|w\|_{L^{2}}\left\|\frac{v}{|x|}\right\|_{L^{2}}
   \le
   \frac{2d}{d-2}
   \|w\|_{L^{2}}\|v\|_{\dot H^{1}}.
 \end{equation*}
 Thus we obtain, for $d\ge3$,
 \begin{equation*}
   |B(v,w)|\le 
   \left(1+\frac{2d}{d-2}\right)\|w\|_{L^{2}}\|v\|_{\dot H^{1}}
   \le
   7 \|w\|_{L^{2}}\|v\|_{\dot H^{1}}.
 \end{equation*}
 Summing up we have proved that
 $B:\dot H^{1}\times L^{2}\to \mathbb{C}$ with norm 1
 and
 $B:L^{2} \times\dot H^{1}\to \mathbb{C}$ with norm $\le7$. 
 By complex bilinear interpolation this implies
 \begin{equation*}
   B:\dot H^{1/2}\times \dot H^{1/2}\to \mathbb{C}
 \end{equation*}
 with norm $\sqrt{7}\le 3$ as claimed.
 Using \eqref{eq:H12} we can write
 \begin{equation*}
   	 \left.
     \Im\int\bar{u}\psi' \partial_{r} u dx\right|_{t=0}^{t=T}
     \le
     4\|u(T)\|_{\dot H^{1/2}}^{2}+4\|u(0)\|_{\dot H^{1/2}}^{2}
     \le
     C_{0}\|u(0)\|_{\dot H^{1/2}}^{2}
 \end{equation*}
 where in the last inequality we used the almost conservation
 law \eqref{eq:energy}, and the constant $C_{0}$
 is independent of $T$. 
 Using the last inequality in \eqref{eq:intxt} we obtain
 \begin{equation*}
   \int_{0}^{T}
   \int
   \left[
   	\frac{\epsilon r^{\epsilon}}{(1+r^{\epsilon})^{2}}
   	\frac{|\nabla u|^{2}}{r}d
   	+
   	\frac{\epsilon r^{\epsilon}}{1+r^{\epsilon}}
   	\frac{|u|^{2}}{r^{3}}
   \right]
   dxdt
   \le
   C_{0}\|u(0)\|_{\dot H^{1/2}}^{2}
 \end{equation*}
 and letting $T\to +\infty$ we arrive at
 \eqref{eq:firstest}.

 In order to prove \eqref{eq:secondest}, we isolate the first term
 in \eqref{eq:firstest} and we use again Theorem \ref{mainthm}
 in the unweighted case:
 \begin{equation*}
 	\int \int
   \frac{ r^{\epsilon-1}}{(1+r^{\epsilon})^{2}}
   |\nabla e^{it \mathcal{L}_{a}}f|^{2}
   dxdt
   \le
   C \epsilon^{-1}\|(-\Delta)^{1/4}f\|_{L^{2}}^{2}
   \simeq
   \epsilon^{-1}\|\mathcal{L}_{a}^{1/4}f\|_{L^{2}}^{2}.
 \end{equation*}
 We note that the weight 
 $w(x)=\frac{ r^{\epsilon-1}}{(1+r^{\epsilon})^{2}}=
   \frac{ |x|^{\epsilon-1}}{(1+|x|^{\epsilon})^{2}}$
 is an $A_{2}$ weight (with $[\cdot ]_{A_{2}}$ norm uniformly
 bounded for $\epsilon\in(0,1)$), hence we can replace
 $\nabla$ with $(-\Delta)^{1/2}$ at the left hand side.
 Moreover, $w(x)$
 satisfies the conditions of Theorem \ref{mainthm},
 so that we can write
 \begin{equation*}
   \int
     \frac{ r^{\epsilon-1}}{(1+r^{\epsilon})^{2}}
     |\mathcal{L}^{1/2}_{a} e^{it \mathcal{L}_{a}}f|^{2}
    dx
   \lesssim
   	\int
      \frac{ r^{\epsilon-1}}{(1+r^{\epsilon})^{2}}
      |\nabla e^{it \mathcal{L}_{a}}f|^{2}
      dx
    \lesssim
    \epsilon^{-1}\|\mathcal{L}_{a}^{1/4}f\|_{L^{2}}^{2}.
 \end{equation*}
 Since $\mathcal{L}_{a}^{1/4}$ commutes with the flow, this implies
 \begin{equation*}
   \int
     \frac{ r^{\epsilon-1}}{(1+r^{\epsilon})^{2}}
     |\mathcal{L}^{1/4}_{a} e^{it \mathcal{L}_{a}}f|^{2}
    dx
   \lesssim
    \epsilon^{-1}\|f\|_{L^{2}}^{2}.
 \end{equation*}
 Finally, again by Theorem \ref{mainthm}, we have
 \begin{equation*}
   \int
     \frac{ r^{\epsilon-1}}{(1+r^{\epsilon})^{2}}
     |\mathcal{L}^{1/4}_{a} e^{it \mathcal{L}_{a}}f|^{2}
    dx
   \gtrsim
   \int
     \frac{ r^{\epsilon-1}}{(1+r^{\epsilon})^{2}}
     |(-\Delta)^{1/4} e^{it \mathcal{L}_{a}}f|^{2}
    dx
 \end{equation*}
 and this gives \eqref{eq:secondest}.

 Denote now by $R(z)$ the resolvent operator of $\mathcal{L}_{a}$
 and by $\Im R(z)$ its imaginary part:
 \begin{equation*}
   R(z)=(\mathcal{L}_{a}-z)^{-1},
   \qquad
   \Im R(z)=(2i)^{-1}(R(z)-R(\bar{z})).
 \end{equation*}
 Moreover, let $A$ be the operator
 \begin{equation*}
   A=w(x)^{1/2}(-\Delta)^{1/4},
   \qquad
   w(x)=\frac{ |x|^{\epsilon-1}}{(1+|x|^{\epsilon})^{2}}.
 \end{equation*}
 Estimate \eqref{eq:secondest} can be written
 \begin{equation*}
   \|A e^{it \mathcal{L}_{a}}f\|_{L^{2}(\mathbb{R}^{d+1})}
   \lesssim
   \epsilon^{-1/2}\|f\|_{L^{2}}.
 \end{equation*}
 By Kato smoothing theory, 
 applying e.g.~Theorem 2.2 in \cite{D}, we obtain that
 this estimate is equivalent to the resolvent estimate
 \begin{equation}\label{eq:resolv}
   \|A \Im R(z)A^{*}f\|_{L^{2}(\mathbb{R}^{d})}\lesssim
   \epsilon^{-1}\|f\|_{L^{2}},
   \qquad
   z\not\in \mathbb{R},
 \end{equation}
 uniformly in $z\not\in \mathbb{R}$.
 (In the terminology of Kato's theory, the closed
 operator $A$ is $\mathcal{L}_{a}$-\emph{smoothing}).
 Then we are in position to apply Theorem 2.4 from \cite{D} 
 (with $\nu=0$) and we obtain that the operator 
 $A \mathcal{L}_{a}^{-1/4}$ is $\mathcal{L}_{a}^{1/2}$-smoothing, 
 i.e., the following estimate holds:
 \begin{equation*}
   \|Ae^{it\mathcal{L}_{a}^{1/2}}f\|_{L^{2}(\mathbb{R}^{d+1})}
   \lesssim \epsilon^{-1/2}\|\mathcal{L}_{a}^{1/4}f\|_{L^{2}}
 \end{equation*}
 that is to say, we have proved that
 \begin{equation*}
   \|w(x)^{1/2} (-\Delta)^{1/4}
     e^{it\mathcal{L}_{a}^{1/2}}f\|_{L^{2}(\mathbb{R}^{d+1})}
   \lesssim \epsilon^{-1/2}\|\mathcal{L}_{a}^{1/4}f\|_{L^{2}}.
 \end{equation*}
 Finally, using Theorem \ref{mainthm} exactly as in the proof
 of \eqref{eq:secondest}, we can cancel the operators
 $(-\Delta)^{1/4}$ and $\mathcal{L}_{a}^{1/4}$, and
 we obtain estimate \eqref{eq:thirdest}.
\end{proof}

\bigskip
{\bf Acknowledgement.} The first-named and the third-named authors were supported by the research grant ARC DP140100649 from the Australian Research Council.  The fourth-named author is supported by  the research grant ARC 170101060 from the Australian Research Council and by Macquarie University New Staff Grant. The  authors would like to thank the referees for useful comments and suggestions to improve the paper.


\begin{thebibliography}{999}
\label{Sect:Bibliography}

\bibitem{AM} P. Auscher, J.M. Martell, Weighted norm inequalities, off-diagonal estimates and elliptic
operators. Part I: General operator theory and weights,
Adv. in Math. 212 (2007), 225--276.

\bibitem{BRV} J.-A. Barcel\'{o}, A. Ruiz, and L. Vega, Some dispersive estimates for Schr\"odinger equations with repulsive potentials, {J. Funct. Anal.} 236 (2006), 1--24.




\bibitem{BZ} F. Bernicot, and J. Zhao, New abstract Hardy spaces, J. Funct.
	Anal. 255 (2008), 1761--1796.

\bibitem{Bu1} N. Burq, F. Planchon, J. Stalker, and A. S. Tahvildar-Zadeh, Strichartz estimates for the wave and Schr\"odinger equations with the inverse-square potential, J. Funct. Anal. 203 (2003), 519--549.

\bibitem{Bu2} N. Burq, F. Planchon, J. Stalker, and A. S. Tahvildar-Zadeh, Strichartz estimates for the wave and Schr\"odinger equations with potentials of critical decay, Indiana Univ. Math. J. 53 (2004), 1665--1680.

\bibitem{CD} F. Cacciafesta and P. D'Ancona, Weighted $L^p$ estimates for powers of selfadjoint operators, Adv. Math. 229 (2012), no. 1, 501--530.

\bibitem{D} P. D'Ancona, Kato smoothing and Strichartz estimates for wave equations with magnetic potentials, Comm. Math. Phys. 335 (2015), no. 1, 1--16.


 \bibitem{DF} P. D'Ancona and L Fanelli, Smoothing estimates for the {S}chr\"odinger equation with unbounded potentials, {J. Differential Equations} 246 (2009), 4552--4567

 \bibitem{DL} P. D'Ancona and R. Luc\`{a}, {Stein--Weiss and Caffarelli--Kohn--Nirenberg inequalities with angular integrability}, {{J. Math. Anal. Appl.}} 388 (2012), 1061--1079

 \bibitem{DR} P {D'Ancona} and R. {Racke}, {Evolution equations on non-flat waveguides}, {{Arch. Ration. Mech. Anal.}} 206 (2012), 81--110

\bibitem{Da} E.B. Davies, {\it Heat kernels and spectral theory}.
Cambridge Univ. Press, 1989.

\bibitem{Du} J. Duoandikoetxea, {\it Fourier Analysis}, Grad. Stud. math, 29, American Math. Soc., Providence,
2000.

\bibitem{DM} X. T. Duong and A. McIntosh, Singular integral operators with non-smooth kernels on
irregular domains, Rev. Mat. Iberoamericana 15 (1999),
233--265.

\bibitem{Fa} L. Fanelli, V. Felli, M.A. Fontelos and A. Primo, Time decay of scaling critical electromagnetic Schr\"odinger flows, Comm. Math. Phys. 324 (2013), 1033--1067.

\bibitem{FV} L. Fanelli and L. Vega, Magnetic virial identities, weak dispersion and Strichartz inequalities, Math. Ann. 344 (2009), no. 2, 249--278.
 
 
\bibitem{GF} J. Garcia-Cuerva and J. L. Rubio de Francia, Weighted norm inequalities and related topics, North-Holland Math. Stud., Amsterdam, 1985.

\bibitem{GVV} M. Goldberg, L. Vega and N. Visciglia, Counterexamples of Strichartz inequalities for Schr\"dinger equations with repulsive potentials, Int. Math. Res. Not. IMRN 2006 (2006) 13927.

\bibitem{HL} A. Hassell and P. Lin, The Riesz transform for homogeneous Schr\"odinger operators on metric
cones, Rev. Mat. Iberoam. 30 (2014), no. 2, 477--522.

\bibitem{IK} A. Ionescu and C. Kenig, Well-posedness and local smoothing of solutions of Schr\"odinger equations, Math.
Res. Lett. 12 (2005), 193--205.

\bibitem{JN} R. Johnson and C.J. Neugebauer, Change of variable results for $A_p$-and reverse H\"older
$RH_r$-classes, Trans. Amer. Math. Soc. 328 (1991), no. 2, 639--666.


\bibitem{Ka} H. Kalf, U.W. Schmincke, J. Walter and R. W\"ust, On the spectral theory of Schr\"odinger and Dirac operators with strongly singular potentials, in: Spectral Theory and Differential Equations, in: Lecture Notes in Math., vol.448, Springer, Berlin, 1975, pp.182--226.

\bibitem{KY} T. Kato and K. Yajima, Some examples of smooth operators and the associated smoothing effect, Rev. Math. Phys. 1 (1989), 481--496.


\bibitem{K.et.al} R. Killip, C. Miao, M. Visan, J. Zhang and J. Zheng, Sobolev spaces adapted to the Schr\"odinger operator with inverse-square potential, available at http://arxiv.org/abs/1503.02716v2.

\bibitem{K.et.al 2} R. Killip, C. Miao, M. Visan, J. Zhang and J. Zheng, The energy-critical NLS with inverse-square potential, available at http://arxiv.org/abs/1509.05822.


\bibitem{LS} V. Liskevich and Z. Sobol, Estimates of integral kernels for semigroups associated with second
order elliptic operators with singular coefficients, Potential Anal. 18 (2003), 359--390.

\bibitem{MZZ} C. Miao, J. Zhang and J. Zhang, Maximal
estimates for Schr\"odinger equation with inverse-square potential, Pacific Journal of Mathematics 273 (2015), 1--19.

\bibitem{Ma} J. Marzuola, J. Metcalfe and D. Tataru, Strichartz estimates and local smoothing estimates for asymptotically flat Schr\"odinger equations, J. Funct. Anal. 255 (2008), 1497--1553.

\bibitem{Mc} A. McIntosh, Operators which have an $H_\vc$-calculus, Miniconference on
operator theory and partial differential equations, Proc. Centre
Math. Analysis, ANU, Canberra, 14 (1986), 210--231.


\bibitem{MS} P. D. Milman and Y. A. Semenov, Global heat kernel bounds via desingularizing
weights, J. Funct. Anal. 212 (2004), 373--398.

\bibitem{PST} F. Planchon, J. Stalker and A.S. Tahvildar-Zadeh, $L^p$ estimates for the wave equation with the inverse-square potential, Discrete Contin. Dyn. Syst. 9 (2003), 427--442.

\bibitem{RS} I. Rodnianski and W. Schlag, Time decay for solutions of Schr\"odinger equations with rough and time-dependent potentials, Invent. Math. 155 (2004), 451--513.

\bibitem{St} E. M. Stein, Harmonic Analysis: Real--variable Methods, Orthogonality, and Oscillatory Integrals, Princeton University Press, Princeton, N. J., 1993.

\bibitem{T} E.C. Titchmarsh, \emph{Eigenfunction expansions associated with second-order differential equations}. University Press, Oxford, 1946.

\bibitem{VZ} J. L. Vazquez and E. Zuazua, The Hardy inequality and the asymptotic behaviour of the heat equation with an inverse-square potential, J. Funct. Anal. 173 (2000), 103--153.

\bibitem{ZZ} J. Zhang and J. Zheng, Scattering theory for nonlinear Schr\"odinger with inverse-square potential, J. Funct. Anal. 267 (2014), 2907--2932.

\end{thebibliography}
\end{document}